\documentclass[11pt,reqno,oneside]{amsart}
\usepackage[top=1.2in, bottom=1in, left=1.2in,right=1.2in]{geometry}

\usepackage{amsmath,amsfonts,amssymb,amscd,amsthm,amsbsy,amscd,bbm, epsf,calc,enumerate,color,graphicx,verbatim,cite,setspace}
\usepackage[foot]{amsaddr}
\usepackage{tikz}
\usepackage{tikz-cd}
\usetikzlibrary{arrows,matrix,positioning}
\usetikzlibrary{decorations.pathreplacing,decorations.pathmorphing,calc}
\usepackage{tkz-graph}
\usepackage{thmtools}
\usepackage{thm-restate}
\usepackage[margin=1cm]{caption}
\usepackage{mathtools}
\usepackage[colorlinks=true, linkcolor=blue, citecolor=red, pdfborder={0 0 0}]{hyperref}
\usepackage{cleveref}
\usepackage{accents}
\usepackage{upgreek}
\usepackage{wrapfig}
\usepackage{multirow}
\usepackage{float}
\usepackage{blkarray}

\usepackage{setspace}

\theoremstyle{plain}
 \newtheorem{theorem}{Theorem}[section]
 \newtheorem{lemma}[theorem]{Lemma}
 \newtheorem{proposition}[theorem]{Proposition}
 \newtheorem{corollary}[theorem]{Corollary}
 \newtheorem{conjecture}[theorem]{Conjecture}
 \newtheorem{claim}[theorem]{Claim}
 
 \newtheorem{observation}[theorem]{Observation}

\theoremstyle{def}
 \newtheorem{definition}[theorem]{Definition}
\theoremstyle{remark}
\newtheorem{remark}[theorem]{Remark}

\makeatletter
\newenvironment{proofof}[1]{\par
  \pushQED{\qed}%
  \normalfont \topsep6\p@\@plus6\p@\relax
  \trivlist
  \item[\hskip\labelsep
        \itshape
    Proof of #1\@addpunct{.}]\ignorespaces
}{%
  \popQED\endtrivlist\@endpefalse
}
\makeatother



\numberwithin{equation}{section}

\newcommand\nc\newcommand
\newcommand\dmo\DeclareMathOperator

\dmo{\re}{Re}
\dmo{\im}{Im}
\dmo{\sign}{sign}
\dmo{\spt}{spt}
\dmo{\supp}{supp}
\dmo{\sym}{Sym}
\dmo{\R}{\mathbb{R}}
\dmo{\C}{\mathbb{C}}
\dmo{\N}{\mathbb{N}}
\dmo{\Z}{\mathbb{Z}}
\dmo{\Id}{{Id}}
\nc{\erdos}{Erd\H os }
\nc{\er}{Erd\H os--R\'enyi } 

\def \<{\langle}
\def \>{\rangle}

\nc{\norm}[1]{\left \| #1 \right \|}
\nc{\expo}[1]{\exp \left( #1 \rule{0mm}{3mm}\right)}
\nc{\cond}{\left\vert\vphantom{\frac{1}{1}}\right.}
\DeclarePairedDelimiter\parentheses{\lparen}{\rparen}
\newcommand{\harp}{\overset{\rightarrow}}
\newcommand{\garp}{\overset{\leftarrow}}
\dmo{\e}{\mathbb{E}}
\dmo{\var}{Var}
\dmo{\pr}{\mathbb{P}}
\dmo{\un}{\mathbbm{1}}
\nc{\eqd}{\,{\buildrel d \over =}\,}
\dmo{\ber}{Ber}
\dmo{\sber}{Ber_{\pm}}
\dmo{\bin}{Bin}
\nc{\bad}{\mathcal{B}}
\nc{\event}{\mathcal{E}}
\nc{\good}{\mathcal{G}}
\nc{\pro}[1]{\mathbb{P}\parentheses*{#1 \rule{0mm}{0mm}}}
\nc{\pros}[2][]{\mathbb{P}_{#1} \parentheses*{ #2 \rule{0mm}{3mm}}}
\nc{\ero}[1]{\mathbb{E}\parentheses*{ #1 \rule{0mm}{3mm}}}
\nc{\eros}[2][]{\mathbb{E}_{#1} \parentheses*{ #2 \rule{0mm}{3mm}}}
\nc{\set}[1]{\left\{ #1 \right\}}

\dmo{\tr}{tr}
\dmo{\rank}{rank}
\dmo{\corank}{corank}
\def \tran {\mathsf{T}}

\nc{\mat}[4][]{\begin{pmatrix} #1 & #2 \\ #3 & #4\end{pmatrix}}
\dmo{\idd}{\mat[1]{0}{0}{1}}
\dmo{\jdd}{\mat[0]{1}{1}{0}}
\nc{\twov}[2][]{\begin{pmatrix} #1 \\ #2 \end{pmatrix}}
\dmo{\eone}{\twov[1]{0}}
\dmo{\etwo}{\twov[0]{1}}
\nc{\Span}{\operatorname{span}}

\nc{\eps}{\varepsilon}
\nc{\ep}{\epsilon}

\nc{\mA}{\mathcal{A}}
\nc{\mB}{\mathcal{B}}
\nc{\mC}{\mathcal{C}}
\nc{\mD}{\mathcal{D}}
\nc{\mE}{\mathcal{E}}
\nc{\mF}{\mathcal{F}}
\nc{\mG}{\mathcal{G}}
\nc{\mH}{\mathcal{H}}
\nc{\mI}{\mathcal{I}}
\nc{\mJ}{\mathcal{J}}
\nc{\mK}{\mathcal{K}}
\nc{\mL}{\mathcal{L}}
\nc{\mM}{\mathcal{M}}
\nc{\mN}{\mathcal{N}}
\nc{\mO}{\mathcal{O}}
\nc{\mP}{\mathcal{P}}
\nc{\mQ}{\mathcal{Q}}
\nc{\mR}{\mathcal{R}}
\nc{\mS}{\mathcal{S}}
\nc{\mT}{\mathcal{T}}
\nc{\mU}{\mathcal{U}}
\nc{\mV}{\mathcal{V}}
\nc{\mW}{\mathcal{W}}
\nc{\mX}{\mathcal{X}}
\nc{\mY}{\mathcal{Y}}
\nc{\mZ}{\mathcal{Z}}

\nc{\tx}{\tilde{x}}
\nc{\ty}{\tilde{y}}
\nc{\tpi}{\tilde{\pi}}

\DeclareMathSymbol{\wtil}{\mathord}{largesymbols}{"65}
\newcommand\lowerwtil{%
  \text{\smash{\raisebox{-1.3ex}{%
    $\wtil$}}}}
\newcommand\witil[1]{%
  \mathchoice
    {\accentset{\displaystyle\lowerwtil}{#1}}
    {\accentset{\textstyle\lowerwtil}{#1}}
    {\accentset{\scriptstyle\lowerwtil}{#1}}
    {\accentset{\scriptscriptstyle\lowerwtil}{#1}}
}

\nc{\tM}{\witil{M}}
\nc{\tR}{\witil{R}}
\nc{\tS}{\witil{S}}
\nc{\tT}{\witil{T}}
\nc{\tX}{\witil{X}}
\nc{\tY}{\witil{Y}}
\nc{\tZ}{\witil{Z}}
\nc{\tA}{\witil{A}}
\nc{\tD}{\witil{D}}

\nc{\tevent}{{\witil{\event}}}
\nc{\tbad}{\witil{\bad}}

\nc{\hx}{\hat{x}}
\nc{\hy}{\hat{y}}
\nc{\hz}{\hat{z}}

\dmo{\degr}{deg}

\dmo{\ones}{\mathbf{1}}
\nc{\pair}{\mathbf{p}}
\nc{\qair}{\mathbf{q}}
\nc{\eye}{{\mathbf{I}_2}}
\nc{\jay}{{\mathbf{J}_2}}
\nc{\kay}{{\mathbf{K}_2}}
\nc{\rowpair}{(i_1,i_2)}
\dmo{\sls}{sls}
\dmo{\SLS}{SLS}
\dmo{\codeg}{codeg}
\dmo{\discrep}{discrep}
\dmo{\co}{co}
\dmo{\Co}{Co}
\dmo{\Ex}{Ex}
\dmo{\ex}{ex}
\nc{\x}{k}
\dmo{\exbar}{\overline{ex}}
\nc{\nbs}[4][]{\mN^{#1}_{(#2,#3)}(#4)}
\nc{\hmR}{\hat{\mR}}
\dmo{\Steps}{Steps}
\dmo{\steps}{steps}
\dmo{\Flats}{Flats}
\dmo{\Perm}{Perm}
\dmo{\disc}{disc}
\nc{\pe}{p}
\dmo{\simple}{simple}
\dmo{\reg}{reg}

\tikzset{
	My Style/.style={
        circle,
        draw,
        fill          = black!50,
        inner sep     = 0pt,
        minimum width =4 pt
    }   
}
\newcommand{\eyepic}{
\begin{tikzpicture}[thick,scale=1.9,->,
                   shorten >=2pt+0.5*\pgflinewidth,
                   shorten <=2pt+0.5*\pgflinewidth,
                        ]
\path[draw] 
       node [My Style] at (.2,0) {}  
       node [My Style] at (1.2,1) {} 
       node [My Style] at (.2,1) {}  
       node [My Style] at (1.2,0) {} ;
        \node at (0,0) {$i_2$};
        \node at (0,1) {$i_1$};
        \node at (1.4,0) {$j_2$};
        \node at (1.4,1) {$j_1$};

\begin{scope}	[arrows={-angle 60}]
    \draw (.2,0) -- (1.2,0) ; 
    \draw (.2,1) -- (1.2,1) ;
\end{scope}

\begin{scope}   [arrows={-angle 60}, dashed]   
     \draw (.2,0) -- (1.2,1)  ; 
     \draw (.2,1) -- (1.2,0)  ;
\end{scope}

\end{tikzpicture}
} 

\newcommand{\jaypic}{
\begin{tikzpicture}[thick,scale=1.9,->,
                   shorten >=2pt+0.5*\pgflinewidth,
                   shorten <=2pt+0.5*\pgflinewidth,
                        ]
\path[draw] 
       node [My Style] at (.2,0) {}  
       node [My Style] at (1.2,1) {} 
       node [My Style] at (.2,1) {}  
       node [My Style] at (1.2,0) {} ;
        \node at (0,0) {$i_2$};
        \node at (0,1) {$i_1$};
        \node at (1.4,0) {$j_2$};
        \node at (1.4,1) {$j_1$};

\begin{scope}	[arrows={-angle 60}]

     \draw (.2,0) -- (1.2,1)  ; 
     \draw (.2,1) -- (1.2,0)  ;
\end{scope}

\begin{scope}   [arrows={-angle 60}, dashed]  
    \draw (.2,0) -- (1.2,0) ; 
    \draw (.2,1) -- (1.2,1) ;     
\end{scope}

\end{tikzpicture}
}

\title{Discrepancy properties for random regular digraphs}
\author{Nicholas A. Cook}
\address{Department of Mathematics, UCLA, Los Angeles, CA 90095-1555}
\email{nickcook@math.ucla.edu}

\begin{document}

\maketitle

\begin{abstract}
For the uniform random regular directed graph we prove concentration inequalities for (1) codegrees and (2) the number of edges passing from one set of vertices to another.
As a consequence, we can deduce discrepancy properties for the distribution of edges essentially matching results for \er digraphs obtained from Chernoff-type bounds.
The proofs make use of the method of exchangeable pairs, developed for concentration of measure by Chatterjee in \cite{Chatterjee}.
Exchangeable pairs are constructed using two involutions on the set of regular digraphs: a well-known ``simple switching" operation, as well as a novel ``reflection" operation. 
\end{abstract}

\section{Introduction}			\label{sec_intro}

For $n\ge 1$ and $d\in [n]=\set{1,\dots,n}$, let $\mD_{n,d}$ denote the set of $d$-regular directed graphs on $n$ labeled vertices -- that is, with each vertex having $d$ in-neighbors and $d$ out-neighbors (allowing self-loops). Let $\Gamma=(V,E)$ be a uniform random element of $\mD_{n,d}$. 
One may identify $\Gamma$ with a uniform random $d$-regular bipartite graph on $n+n$ vertices in the obvious way. 
We will stick with the digraph interpretation, though we note that all of our results can be extended to cover $(d,d')$-regular bipartite graphs on $m+n$ vertices; see \Cref{sec_extend}.

Our aim in this paper is to show that two types of statistics of $\Gamma$ are sharply concentrated when $n$ is large and $d$ is sufficiently large depending on $n$. 
We identify $V$ with $[n]$ throughout, and view $E$ as a subset of $[n]^2$. We sometimes write $i\rightarrow j$ to mean $(i,j)\in E$.
\begin{enumerate}
\item {\bf Codegrees:} Denote the number of common out-neighbors of a fixed pair of vertices $i_1,i_2\in [n]$ by
$$\harp{\co}_\Gamma(i_1,i_2):=\left|\Big\{ j\in [n]: i_1\rightarrow j \mbox{ and } i_2\rightarrow j\Big\}\right|$$
and the number of common in-neighbors by
$$\garp{\co}_\Gamma(i_1,i_2):=\left|\Big\{ j\in [n]: i_1\leftarrow j \mbox{ and } i_2\leftarrow j\Big\}\right|.$$
We expect these statistics to be of size roughly $d^2/n=\pe^2n$, where we denote by $\pe:=d/n$ the average edge density for $\Gamma$.\\

\item {\bf Edge counts:} For fixed subsets of vertices $A,B\subset [n]$, denote the number of edges passing from $A$ to $B$ by
$$e_\Gamma(A,B):= \left| E\cap (A\times B) \right|.$$
We expect this statistic to be of size roughly $\pe|A||B|=: \mu(A,B)$. 
We refer to the deviation
\begin{equation}
\disc_\Gamma(A,B) := \big| e_\Gamma(A,B) - \mu(A,B) \big|
\end{equation}
as the \emph{(edge) discrepancy of $\Gamma$ at $(A,B)$}. 
\end{enumerate}
We will loosely use the term \emph{edge discrepancy property} to refer to a bound on edge discrepancies holding uniformly for all pairs $(A,B)$, or at least for all pairs of ``sufficiently large" sets $A,B$.

The discrepancy properties and control on codegrees proved in the present work were an important component in the recent proof by the author that random 0/1 matrices with constant row and column sum $d$ are invertible with high probability, assuming $\min(d, n-d)=\omega(\log^2n)$; see \cite{Cook:sing}. 
We expect that the results of this paper will also be useful for questions of a more graph-theoretic nature. 

\subsection{Background on random regular graphs}

Random graphs have been studied intensively since their popularization by Erd\H os as a tool for proving of the existence of graphs with certain properties, often when no constructive approach was known (such as graphs with arbitrarily large chromatic number and girth; see \cite{AlSp}).
They have since found myriad applications in computer science, physics, biology, and other fields.
The most commonly used model is the \emph{binomial} or \emph{Erd\H os--R\'enyi} random graph $G(n,p)$, in which each of the ${n\choose 2}$ possible edges is present independently of all others with probability $p$. 
We may similarly define the \er digraph $D(n,p)$, which has $n^2$ possible directed edges.

Random regular graphs emerged as a popular model much later, and their origin can also be traced back to a question in extremal combinatorics: are there expander graphs of bounded degree?
(Strictly speaking the term ``expander" only makes sense for a sequence of graphs; the reader may consult the survey \cite{HLW} for a precise statement of this question.)
This was answered in the affirmative by Pinsker in 1973 \cite{Pinsker} (and independently by Barzdin and Kolmogorov in the bipartite case \cite{BaKo}) who showed that certain random regular graphs of constant degree are expanders with positive probability.

Since then, much of the interest in random regular graphs has been due to their robust connectivity properties as compared to \er graphs. 
Indeed, while \er graphs are asymptotically almost surely \emph{disconnected} when the average degree is smaller than $\log n$, 
random regular graphs are not only connected with high probability for degree as small as $3$, they are \emph{nearly Ramanujan} (meaning they are near-optimal expanders in a certain sense; see \cite{Friedman}).

Random regular graphs are often harder to analyze than their binomial counterparts since the $d$-regularity constraint destroys the independence of the edges.
Nevertheless, asymptotic enumeration results were obtained in \cite{BeCa}, \cite{Bollobas} and \cite{Wormald:problems}.
The introduction by Bollob\'as in \cite{Bollobas} of the \emph{configuration model} for the uniform random regular graph allowed for many later developments (ideas similar to the configuration model were also present in \cite{BeCa} and \cite{Wormald:problems}).
Here one generates a uniform random regular graph by the following procedure:
\begin{enumerate}
\item Associate to each vertex $v\in V$ a ``fiber" $F_v$ of $d$ ``points", so that there are 
$$\bigg| \bigcup_{v\in V} F_v \bigg| = nd$$ 
points in total. 
\item Select a pairing $\mP$ of the $nd$ points uniformly at random.
\item Now collapse each fiber $F_v$ to the associated vertex $v$: we say that $v$ is connected to $w$ if there are points $v'\in F_v$, $w'\in F_w$ such that $v'w'\in\mP$. 
In general the resulting graph $G=G(\mP)$ is a $d$-regular \emph{multi-graph}; however, conditional on the event $\event_{\simple}$ that that $\mP$ collapses to a simple graph, it is easy to check that $G(\mP)$ is a uniform random $d$-regular graph. Hence we may
\item repeat this process if necessary until we obtain a simple graph.
\end{enumerate}
The procedure can be modified to generate uniform random $d$-regular directed or bipartite graphs in the obvious manner.
When using the configuration model to bound the probability of an event $\bad$ holding for a uniform random regular graph, one ``lifts" $\bad$ to the corresponding event $\bad'$ for the pairing $\mP$, which is often easier to analyze. 
Then one can bound
\begin{equation}		\label{boundsimple}
\pro{\bad} = \pro{\bad' | \event_{\simple}} \le \frac{\pro{\bad'}}{\pro{\event_{\simple}}}.
\end{equation}

A particularly nice feature is that working with the random pairing $\mP$ rather than with the graph $G$ gives access to concentration of measure inequalities for martingale sequences (e.g.\ the Azuma--Hoeffding inequality).
(However, we will see below that the method of exchangeable pairs can be applied directly to the uniform measure on random regular graphs.)
A drawback is that $\pro{\event_{\simple}}$ becomes quite small when the degree $d$ is large. 
Indeed, the enumeration result of \cite{BeCa} implies the estimate
\begin{equation}		\label{psimple}
\pro{\event_{\simple}} = \expo{ -\Theta(d^2)}
\end{equation}
for $d$ fixed (see \Cref{sec_notation} for definitions of asymptotic notation used in this paper).
This asymptotic was later shown by McKay and Wormald in \cite{McWo81} to hold when $d=o(\sqrt{n})$. 

An advantage of the concentration results given in \Cref{thm_main} below is that they are proved for the uniform random regular digraph directly, rather than through the configuration model, and hence do not have to compete with with the small probability in \eqref{psimple}. 
In particular, our results are not limited to $d=o(\sqrt{n})$ (in fact the bounds are strongest for dense graphs).

In a similar spirit to the configuration model, it is possible to deduce some properties of random regular graphs from known results for \er graphs.
Let us consider the case of digraphs. 
With $\Gamma$ a uniform random $d$-regular digraph, draw $D$ from $D(n,p)$ with $p=d/n$, and let $\event_{\reg}$ be the event that $D$ is a $d$-regular graph. 
Note that
$$D\big| \event_{\reg} \eqd \Gamma.$$
We have the asymptotic lower bound
\begin{equation}	\label{psimp}
\pr(\event_{\reg}) = \expo{-\Omega\Big(n \log\big[ \min(d,n-d)\big]\Big)}
\end{equation}
which follows from an asymptotic formula for the number of $d$-regular digraphs on $n$ vertices, established for the sparse case $d=np=o(\sqrt{n})$ by McKay and Wang in \cite{McWa} and for the dense range $\min(d,n-d) \gg n/\log n$ by Canfield and McKay in \cite{CaMc}.
Although enumeration results for $\sqrt{n}\ll d\ll n/\log n$ are unavailable as of this writing (though it is natural to conjecture that the formula \eqref{psimp} extends to hold in this range), in \cite{Tran} Tran used an argument from \cite{ShUp} of Shamir and Upfal to show that for $d=\Omega(\log n)$, 
\begin{equation}		\label{psimplb}
\pr(\event_{\reg}) \ge \expo{-O\big(n\sqrt{d}\big)}.
\end{equation}
Similarly to \eqref{boundsimple} we hence have that
\begin{equation}	\label{drestriction}
\pro{\bad} \le \frac{\pro{\bad \mbox{ holds for $D$}}}{\pro{\event_{\reg}}} \le \expo{O\big(n\sqrt{d}\big)} \pro{\bad \mbox{ holds for $D$}}
\end{equation}
for $d=\Omega(\log n)$ by \eqref{psimplb} (and for $d=o(\sqrt{n})$ or $\min(d,n-d)=\omega(n/\log n)$ we may instead use \eqref{psimp}). 
We refer to this approach as the \emph{restriction strategy}, as it views the uniform measure on the set of $d$-regular graphs as the restriction of a product measure on the full space of graphs. 
With \eqref{drestriction} one is limited to importing properties for random regular graphs which hold with probability $1-O(\exp(-Cn\sqrt{d}))$ for some sufficiently large $C$ for graphs in $D(n,d/n)$. 
We note in particular that the results of the present work deal with events that are too large to be controlled by the restriction method.

In order to go beyond restriction of product measures, 
we must make use of some properties of random $d$-regular graphs besides the crude parameter of edge density $p=d/n$. 
We would like to show that the events
\begin{equation}	\label{independs}
\event_{\reg} \quad \mbox{ and } \quad \big\{\bad \mbox{ holds for $D$}\big\} 
\end{equation}
are approximately independent in some sense.
Indeed, the bound \eqref{drestriction} assumes the worst case that $\set{\bad \mbox{ holds for $D$}}\subset \event_{\reg}$. 
(Note that correlation inequalities such as the FKG bound cannot be applied in this setting as $\event_{\reg}$ is not monotone.)

A natural step in this direction is to understand symmetries of the set of regular graphs -- with a slight abuse of notation we denote this set by $\event_{\reg}$.
Focusing on symmetries of a ``local nature", i.e.\ ones that change only a small number of edges, leads naturally to the method of \emph{switchings}, developed by McKay and Wormald in several works (see the survey \cite{Wormald}). 
For regular digraphs, perhaps the most obvious symmetry is to change between the following two configurations of edges at fixed vertices $i_1,i_2,j_1,j_2$:
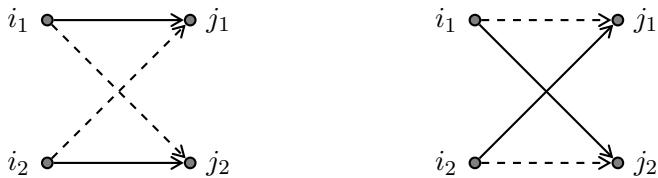
\begin{center}

\eyepic \hspace{2.2cm} \jaypic

\end{center}
where we use a solid arrow to depict an edge and a dashed arrow to indicate the absence of an edge.
We refer to this modification as a \emph{simple switching}.
Roughly speaking, the probability that an event $\bad$ holds for random regular graphs can be estimated by performing a switching in a random fashion (such as by sampling the indices $i_1,i_2,j_1,j_2$ at random) and estimating the probability that the graph enters or leaves the event $\bad$ under the application of the switching.
There is a sense in which this approach is showing the approximate independence we sought in \eqref{independs}: we are performing operations which preserve the event $\event_{\reg}$, and seek to show that these tend to disrupt the event $\bad$.

McKay introduced the method of switchings in \cite{McKay2}, and in \cite{McKay} used it to prove bounds on the probability of occurrence of cycles of various length in a random regular graph of bounded degree. Through the trace method this allowed him to deduce that the limiting spectral distribution of the adjacency matrix is that of the infinite $d$-regular tree, now known as the \emph{Kesten--McKay distribution}. (For $d$ tending to infinity with $n$, the spectral distribution is instead governed by the \emph{semi-circle law}, as was proved by Dumitriu and Pal \cite{DuPa} in the sparse regime $d=n^{o(1)}$, and by Tran, Vu and Wang in the general case using the restriction strategy \cite{TVW}.)
Since then, the method has been extended and applied to several problems on random regular graphs, such as to extend the asymptotic enumeration results of \cite{BeCa}, \cite{Bollobas}, \cite{Wormald:problems} to $d=o(\sqrt{n})$ in \cite{McWo81}.
See the survey \cite{Wormald} for more background on switchings. 
See also \cite{Mathoverflow} for a simple illustration of the method for the problem of estimating the probability that a random permutation has a fixed point.

\subsection{Codegrees, edge discrepancy, and pseudo-randomness} 

Parallel to the study of random graphs, there has been a rich literature on \emph{pseudo-random graphs}, which is an imprecise term for deterministic graphs that exhibit properties held by (Erd\H os--R\'enyi) random graphs with high probability.
Systematic research into pseudo-random graphs was initiated by Thomason in \cite{Thomason1}, \cite{Thomason2}, where he introduced the notion of \emph{jumbled graphs} (see \Cref{def_jumbled} below).
In \cite{CGW}, Chung, Graham and Wilson defined \emph{quasi-random graphs} and proved that several ``pseudo-randomness" properties are in fact equivalent. 
See also the survey \cite{KrSu} and Chapter 9 of \cite{AlSp}.

In particular, the works \cite{Thomason1} and \cite{CGW} highlighted a close connection between codegrees and edge discrepancy, the quantities of interest for the present work.
We have the following result from \cite{KSVW} deducing a discrepancy property from uniform control on codegrees, proved earlier for the \er case in \cite{AKS}, and essentially going back to \cite{Thomason1}.
(While the result in \cite{KSVW} was stated for undirected graphs, the following version can be obtained by following similar lines to the proof given there.)

\vspace{0pt}

\begin{lemma}[Pseudorandomness \cite{KSVW}]			\label{lem_pseudo}
Let $\Gamma$ be a fixed element of $\mD_{n,d}$ with the property that for some $\eps>0$ and for every $i_1,i_2\in [n]$ distinct, 
\begin{equation}	\label{cobound}
\max\Big[\; \harp{\co}_\Gamma(i_1,i_2) \;,\; \garp{\co}_{\Gamma}(i_1,i_2)\;\Big] \le (1+\eps)\pe^2n.
\end{equation}
Then for any pair of sets $A,B\subset [n]$ such that $|A|,|B|\ge \frac{1}{\eps}\frac{n}{d} = (\eps p)^{-1}$, we have
\begin{equation}	\label{edisc1}
\left| \frac{e_\Gamma(A,B)}{\pe|A||B|} - 1\right| \le \left[\frac{2\eps n}{\max(|A|, |B|)} \right]^{1/2}.
\end{equation}

\end{lemma}

\begin{remark}

Note that in order to have concentration of $e_\Gamma(A,B)$ at the scale of the mean $\pe |A||B|$, the lemma requires that one of the sets be of size linear in $n$.
\Cref{thm_main} below will allow us to extend this to much smaller sets.

\end{remark}

\Cref{lem_pseudo} can be used to deduce control on edge discrepancy for random regular digraphs holding asymptotically almost surely (a.a.s.), as soon as one can show that \eqref{cobound} holds a.a.s. 
This was the route taken in \cite{KSVW} for the undirected case by Krievelevich, Sudakov, Vu and Wormald, 
who obtained the following concentration result for codegrees in sufficiently dense $d$-regular graphs. 

\begin{theorem}[From Theorem 2.1 in \cite{KSVW}]		\label{thm_ksvw}
Let $G$ be a uniform random $d$-regular undirected graph on $n$ vertices. 
Suppose that 
$$\omega(\sqrt{n\log n})\le d<n-cn/\log n$$
for some constant $c>2/3$. 
Then asymptotically almost surely we have
\begin{equation}		\label{ksvwbd}
\max_{i_1,i_2\in V} \left|\co_G(i_1,i_2) - \frac{d^2}{n}\right| < C\frac{d^3}{n^2} + 6 d\sqrt{\frac{\log n}{n}}
\end{equation}
for some $C>0$ absolute. If $d\ge cn/\log n$ we may take $C$ to be zero.

\end{theorem}

\begin{remark}
Theorem 2.1 in \cite{KSVW} also states some weaker upper bounds on codegrees valid for smaller $d$, which we have omitted. 
\end{remark}

The proof of \Cref{thm_ksvw} divides into two (overlapping) cases.
For $\min(d, n-d) \ge cn/\log n$ the proof uses an asymptotic enumeration formula for dense graphs with given degree sequence, proved in \cite{McWo90}. 
The method of switchings is used for the case $d=o(n)$. 
The proof shows that the estimate $o(1)$ for the probability that \eqref{ksvwbd} fails is in fact $O(n^{-c})$ for some $c>0$ absolute.

\subsection{Results}			\label{sec_results}

We combine variants of the switching method of McKay and Wormald with the method of exchangeable pairs for concentration of measure, as developed by Chatterjee in \cite{Chatterjee}, to prove exponential tail bounds on codegrees and edge discrepancies. 
For edge discrepancies we use the simple switching coupling, reviewed in \Cref{sec_switching}, while for concentration of codegrees we employ a novel (to our knowledge) ``reflection" coupling, described in \Cref{sec_reflection}.

For both $e_\Gamma(A,B)$ and $\harp{\co}_\Gamma(i_1,i_2)$ we are able to prove tail bounds that match (up to constant factors in the exponential) what can be obtained in the \er case using Chernoff bounds (specifically, Bernstein's inequality). 
As a consequence, we can combine our concentration estimates with union bounds to prove discrepancy properties essentially matching those available for \er digraphs. 
We review the (brief and completely standard) proofs of analogous results for the \er case in \Cref{sec_er} for comparison.

It is possible that our approach can be extended to prove similar results for undirected (non-bipartite) random regular graphs, but we do not pursue this matter here. 
It is also likely that our methods can be applied to the study of directed multi-graphs with given (non-constant) degree sequence.

Before stating our main theorem we set up some notation. 
Due to the constraint of $d$-regularity, a deviation of $e_\Gamma(A,B)$ from its mean coincides with an equal deviation of $e_\Gamma(A^c,B^c)$, where we denote $A^c:= [n]\setminus A$. 
Indeed, if $e_\Gamma(A,B)= k$, we have from $d$-regularity that
\begin{align*}
e_\Gamma(A^c,B) &= d|B|-k,\\
e_\Gamma(A,B^c) &=d|A| - k,\\
e_\Gamma(A^c,B^c) &= d(n-|A|-|B|)+k.
\end{align*}
It follows from the last line that for any $t\in \R$, the following identity of events holds:
\begin{equation}	\label{eventid}
\big\{e_\Gamma(A,B) -\mu(A,B) \ge t \big\} = \big\{ e_{\Gamma}(A^c,B^c) - \mu(A^c,B^c) \ge t \big\}.
\end{equation}
It is hence natural to consider deviations of $e_\Gamma(A,B)$ at the scale
\begin{align}
\hat{\mu}(A,B) &:= \min\big( \mu(A,B), \mu(A^c,B^c) \big) \\
&= \pe \min\big(|A||B|, (n-|A|)(n-|B|)\big).
\end{align}
We will often suppress the dependence of $\mu$ and $\hat{\mu}$ on $A,B$.
We also denote
\begin{equation}
\hat{d}:= \min(d,n-d)
\end{equation}
and $\hat{\pe}:= \hat{d}/n$, the minimum of the edge density of $\Gamma = ([n],E)$ and its complement $\Gamma'= ([n], [n]^2\setminus E)$.

Our main theorem can be summarized as follows:
\begin{enumerate}
\item With high probability, codegrees are uniformly close to $p^2n$.
\item Restricted to the (likely) event that all codegrees are roughly $p^2n$, we have concentration with exponential tails for the edge discrepancy at fixed pairs of sets $A,B$.\\
\end{enumerate}

\begin{theorem}[Main theorem]
\label{thm_main}

For $\eta\ge 0$ define the event
\begin{equation}
\good^{\co}(\eta) = \bigg\{ \forall \set{i_1,i_2}\subset[n], \;  \left|\harp{\co}_\Gamma(i_1,i_2) - p^2n\right| \le \eta p(1-p)n \bigg\}.
\end{equation}
We have
\begin{enumerate}
\item (Uniform control on codegrees) For any $\eta\ge 0$, $\good^{\co}(\eta)$ holds except with probability
\begin{equation}
O\left(n^2\hat{d}^2\expo{ - c \eta \min\big\{\hat{d}, \eta n\big\}}\right).
\end{equation}
In particular, for any $K_1>0$ there exists $K_2>0$ such that $\good^{\co}(\eta)$ holds with probability $1-O\big(n^{-K_1}\big)$ if 
\begin{equation}
\eta\ge K_2\max\left\{ \frac{\log n}{\min(d,n-d)}, \sqrt{\frac{\log n}{n}} \right\}.
\end{equation}
\item (Concentration of edge counts) For any $A,B\subset[n]$ and any $\tau\ge 0$, 
\begin{equation}		\label{edgeup}
\pro{  \Big\{e_\Gamma(A,B)-\mu\ge \tau \hat{\mu}   \Big\}   \wedge  \good^{\co}\big(\eta\big)  } 
\le \expo{-\frac{\tau^2\hat{\mu}}{C_1+C_2\tau}}
\end{equation}
provided $\eta \le \min\left(\frac{1}{4},\frac{\tau}{8}\right)$, and
\begin{equation}		\label{edgelow}
\pro{  \Big\{e_\Gamma(A,B)-\mu\le -\tau \hat{\mu}  \Big\} \wedge \good^{\co}\big(\eta\big) } 
\le \expo{-\frac{\tau^2\hat{\mu}}{C_1}}
\end{equation}
provided $\eta \le \frac{\tau}{4}$, where $C_1,C_2>0$ are absolute constants.
In particular, if  $\eta \le \min\left(\frac{1}{4},\frac{\tau}{8}\right)$, we have
\begin{equation}		\label{edge2side}
\pro{  \Big\{ \disc_\Gamma(A,B) \ge \tau \hat{\mu}(A,B)   \Big\}  \wedge \good^{\co}\big(\eta\big) } 
\le 2 \expo{-\frac{\tau^2}{C_1+C_2\tau}\hat{\mu}(A,B)}.\\
\end{equation}

\end{enumerate}

\end{theorem}

\begin{remark}
The proof shows that one may take 
$C_1=64$, $C_2=8$, though we make little effort to optimize these values.
\end{remark}

\begin{remark}\label{rmk_dgrows}
In order to deduce that $e_\Gamma(A,B)$ is within an arbitrarily small fixed multiplicative error of its mean $\mu=p|A||B|$ using \Cref{thm_main}, one must assume $\min(d,n-d)=\omega(\log n)$. Indeed, we want to take $\tau$ as small as we like in \eqref{edge2side}, which requires taking $\eta\le\tau/8$. Now to deduce that $\good^{\co}(\tau/8)$ holds a.a.s.\ from part (1), we must take $\hat{d}=\min(d,n-d) \ge C\tau^{-1}\log n$ for a sufficiently large constant $C$.
See \Cref{thm_edgeperm} below for a result which is valid for $d=O(\log n)$, but for a slightly different model of random regular digraph (the permutation model).
\end{remark}

\begin{remark}[Comparison to the \er case]
For $A, B$ such that $|A|+|B|\le n$ (i.e. such that $\mu(A,B)=\hat{\mu}(A,B)$), the bound \eqref{edge2side} is the same as what one obtains in the \er case from Bernstein's inequality, up to modification of the constants $C_1,C_2$ -- see \Cref{sec_er} and the bound \eqref{eredge}. 
For the case $|A|+|B|> n$ the bound \eqref{edge2side} becomes superior to \eqref{eredge}.
This is due to the identity \eqref{eventid} (which comes from $d$-regularity): if $A, B$ are of size close to $n$, a large deviation of $e_\Gamma(A,B)$ coincides with a \emph{very} large deviation of $e_\Gamma(A^c, B^c)$.
(Of course, the most concentrated statistic of all is $e_\Gamma(V,V)=dn$, which is deterministic, while this random variable has variance $p(1-p)n^2 \asymp n\min(d,n-d)$ in the \er model.)

\end{remark}

Our proof of both parts of \Cref{thm_main} is by the method of exchangeable pairs.
Roughly speaking, to prove concentration of a statistic $f(\Gamma)$ of the random digraph $\Gamma$, the method is to analyze the change in $f$ under a small random change to $\Gamma$.
To prove the concentration of edge counts in part (2) we will use the simple switching operation on digraphs, reviewed in 
\Cref{sec_switching}.
For the concentration of codegrees in part (1) we use an operation on digraphs which we call ``reflection". 
Reflections are less local in nature than simple switchings; the construction is given in \Cref{sec_reflection}.

\Cref{thm_main} can be viewed as an improvement on the deterministic \Cref{lem_pseudo} for the setting of random graphs. 
Like \Cref{lem_pseudo} it deduces some control on edge discrepancy after restricting to a ``good" event on which there is some uniform control on the codegrees. 
The key differences are the following:
\begin{enumerate}
\item Rather than deduce a deterministic bound on edge discrepancy from the control on codegrees as in \Cref{lem_pseudo} (which obtained an essentially optimal bound), \Cref{thm_main} gives much tighter bounds holding with high probability. 
\item The control on codegrees summarized in the event $\good^{\co}(\eta)$ differs in two respects: on the one hand we allow fluctuations at scale $\pe(1-\pe) n$ rather than $\pe^2n$, which is less stringent for sparse graphs, while on the other hand we need both lower and upper bounds. 
\end{enumerate}

One can deduce various discrepancy properties for $\Gamma$ holding with high probability using \Cref{thm_main}, and essentially matching standard discrepancy properties for \er digraphs (since our tail bounds match the bounds \eqref{erco} and \eqref{eredge} for \er digraphs up to constants in the exponential). 
There is flexibility with the range of sets to consider and the tolerance level for edge discrepancy; the choice will be dictated by the application at hand.
We now give one example. 

Suppose one desires to have $e_\Gamma(A,B)$ within a small factor of its expectation $p|A||B|$ for all pairs of sufficiently large sets $A,B$.
The following corollary shows that this is satisfied with high probability.

\begin{corollary}
\label{cor_discrep}
Let $C_0>0$ be a sufficiently large absolute constant.
For $\eps\in (0,1)$, let $\good(\eps)$ denote the event that for all $A,B\subset [n]$ such that
\begin{equation}
|A|,|B| \ge \frac{C_0\log n}{\eps^2 p}
\end{equation}
we have $\disc_\Gamma(A,B) \le \eps \hat{\mu}(A,B)$.
If $\hat{d}=\min(d,n-d) \ge C_0\eps^{-1}\log n$, then $\good(\eps)$ holds except with probability
\begin{equation}
O\left(\expo{ -c\min\left\{\eps \hat{d}, \eps^2 n,  \frac{n}{\eps^2 d} \log^2n \right\} }\right).
\end{equation}

\end{corollary}

\begin{proof}
By the lower bound on $\hat{d}$ and part (1) of \Cref{thm_main} we have
$$\pro{\good^{\co}(\eps)} \ge 1 - \expo{ -c\min\big\{\eps \hat{d}, \eps^2 n\big\} } $$
(taking $C_0$ sufficiently large to beat the polynomial factors).
By abuse of notation we restrict the sample space to $\good^{\co}(\eps)$.
It now suffices to show
\begin{equation}	\label{sts}
\pro{\good(\eps)} \ge 1 
- C\expo{ -c \frac{n\log^2n}{ \eps^2d} }.
\end{equation}
Since 
$$\big\{\disc_M(A,B) \ge \eps \hat{\mu}\big\} = \big\{\disc_M(A^c,B^c) \ge \eps \hat{\mu}\big\} $$
it suffices to consider pairs $(A,B)$ with $|A|+|B|\le n$.
By giving up a factor of 2 we may also assume $|A|\le |B|$. 

Set
\begin{equation}	\label{azero}
a_0= \frac{C_0\log n}{\eps^2p} .
\end{equation}
For $a_0\le a\le b\le n$ let
$$\bad_{ab}(\eps) = \Big\{ \exists A, B\subset[n]: \; |A|=a, |B|=b, \; \disc_M(A,B) \ge \eps\hat{\mu}(A,B)\Big\}.$$
Applying part (2) of \Cref{thm_main} and a union bound (and by our restriction to $\good^{\co}(\eps)$), 
\begin{align}
\pro{\bad_{ab}(\eps)} & \ll {n\choose a}{n\choose b} \expo{-c\eps^2 pab} \label{third1}\\
& \ll \expo{ Cb\log n - c\eps^2 p ab}		\\
&\ll \expo{ -c\eps^2 pab}		\label{fourth1}
\end{align}
where in the last line we used that $a\ge a_0$ and took $C_0$ sufficiently large (adjusting the constant $c$). 
By another union bound, 
\begin{align*}
\pro{ \good^e(\eps)^c} & \ll \sum_{b=a_0}^n \sum_{a=a_0}^b \pro{ \bad_{ab}(\eps)}\\
& \ll \sum_{b=a_0}^n \sum_{a=a_0}^b \expo{ -c\eps^2 pab} \\
&\ll \expo{ -c\eps^2 p a_0^2}
\end{align*}
where in the last line we performed the geometric sums. 
Substituting the expression \eqref{azero} completes the proof.
\end{proof}

\begin{remark}
Note that in going from \eqref{third1} to \eqref{fourth1} we actually only needed 
$$a\ge \frac{C_0}{\eps^2}\log \frac{en}{b}.$$
Hence, we could have taken the wider class
\begin{equation}
\mF(\eps) = \set{ (A,B): \; A, B\subset[n], \; \min\big(|A|,\,|B|\big) \ge \frac{C_0}{\eps^2}\frac{n}{d}\log\frac{en}{\max(|A|,|B|)}}
\end{equation}
which includes some pairs $(A,B)$ where, say, $|A|\asymp1/p $ and $|B|\asymp n$. 
\end{remark}

Next we state a conjecture concerning the singular value distribution for the adjacency matrix of $\Gamma$, which we denote by $M=M_\Gamma$.
Denote the singular values of $M$ by
$$d=\sigma_1(\Gamma) \ge \sigma_2(\Gamma) \ge \cdots \ge \sigma_n(\Gamma)\ge 0$$
(where $\sigma_1(\Gamma)=d$ follows from $d$-regularity and the Cauchy-Schwarz inequality). 
It is well known that control on edge discrepancy follows from a spectral gap.
We recall the notion of a \emph{jumbled} graph, introduced by Thomason \cite{Thomason1} and adapted here to the setting of digraphs.
\begin{definition}		\label{def_jumbled}
Say that a digraph $D=(V,E)$ is \emph{$\alpha$-jumbled} if for all $A,B\subset [n]$ we have
$$\disc_D(A,B):= \big|e_D(A,B) - \mu(A,B) \big| \le \alpha \sqrt{|A||B|}.$$
\end{definition}
It is a straightforward exercise to show that a $d$-regular digraph on $n$ vertices whose adjacency matrix has second singular value $\sigma_2$ is $\sigma_2$-jumbled (see for instance Theorem 2.11 in \cite{KrSu} for the undirected case; the directed case follows similar lines).

\begin{conjecture}
\label{conj_sig2}
Assume $1\le d\le n$. Then asymptotically almost surely, $\sigma_2(\Gamma) =O\big(\sqrt{d}\big).$
In particular, $\Gamma$ is $O(\sqrt{d})$-jumbled.
\end{conjecture}
The singular vector corresponding to $\sigma_1(\Gamma)=d$ is the constant vector $\frac{1}{\sqrt{n}}\ones:=\frac{1}{\sqrt{n}}(1,\dots,1)$. 
By the Courant-Fischer minimax theorem, letting $\mS^n_0$ denote the set of unit vectors orthogonal to $\ones$, we have
\begin{align*}
\sigma_2(\Gamma) &= \sup_{u\in \mS^n_0}\|M_\Gamma u\|\\
& = \sup_{u\in \mS^n_0} \| (\ones\ones^\tran -M_{\Gamma})u\|\\
& = \sup_{u\in \mS^n_0}\|M_{\Gamma'}u\|\\
&=\sigma_2(\Gamma')
\end{align*}
where $\Gamma'$ is the complementary $(n-d)$-regular digraph.
Hence, it suffice to consider $1\le d\le n/2$.
Using \Cref{thm_main} and union bounds, one can show that $\Gamma$ is $O(\sqrt{d})$-jumbled for the dense case $n\ll d\le n/2$, following similar lines to the proof of \Cref{cor_discrep}.
For the sparse case, this approach can only show that $\Gamma$ is $O(\sqrt{d\log n})$-jumbled.

\Cref{conj_sig2} parallels a conjecture of Vu for the undirected case \cite{Vu:discrete}. 
For an undirected graph $G$ with adjacency matrix $M_G$ having real eigenvalues $\lambda_1(G)\ge \cdots\ge \lambda_n(G)$
we simply have 
$$\sigma_2(G)= \lambda(G):=\max\big\{ |\lambda_2(G)|,\, |\lambda_n(G)|\big\}.$$
In \cite{KrSu}, Kahn and Szemer\'edi proved a bound of $O(\sqrt{d})$ for $\lambda(\Pi)$, with $d$ fixed independent of $n$, and with the graph $\Pi$ drawn from a different distribution on random regular graphs which we call the \emph{permutation model}.
Let $P_1,\dots,P_d$ be iid uniform $n\times n$ permutation matrices, and put 
\begin{equation}	\label{Mper}
M_{\Lambda} = P_1 +\cdots +P_d.
\end{equation}
We may interpret $M_{\Lambda}$ as the adjacency matrix for a random $d$-regular directed multi-graph $\Lambda$, and we may also associate $M_{\Lambda}+M_{\Lambda}^\tran$ to a $2d$-regular undirected multi-graph $\Pi$.
Kahn and Szemer\'edi proved that if $d$ is fixed independent of $n$, we have
\begin{equation}	\label{KSsig}
\sigma_2(\Lambda) = O(\sqrt{d})
\end{equation}
asymptotically almost surely. 
By the triangle inequality this implies $\lambda(\Pi)=O(\sqrt{d})$ a.a.s. 
Their argument was later extended to allow $d=o(\sqrt{n})$ in \cite{DJPP}, and was also adapted to the configuration model with $d=o(\sqrt{n})$ in \cite{BFSU}.
Furthermore, the optimal bound $\lambda(\Pi)\le 2\sqrt{2d-1}+o(1)$ was obtained for fixed $d$ by Friedman in \cite{Friedman} by a completely different argument.

For small degree, the permutation model $\Lambda$ is ``close" to the uniform model $\Gamma$ in the following precise sense. 
It was proved in \cite{Janson:contiguity} and \cite{MRRW} that if $d$ is fixed, the models
\begin{enumerate}
\item $\Gamma$ (a uniform random element of $\mD_{n,d}$), and
\item $\Lambda$ conditioned to be simple
\end{enumerate}
are \emph{contiguous}, meaning that a sequence of events holding a.a.s.\ for one model will hold a.a.s.\ for the other.
In particular, for the case that $d$ is fixed \Cref{conj_sig2} follows from contiguity and the bound \eqref{KSsig}.
It was also shown that the model $\Pi$ is contiguous to a uniform random regular graph of fixed even degree.
We believe that these models continue to be contiguous if $d=O(\log n)$, though we are not aware of any such results in the literature. 

We record an analogue of our main theorem for the permutation model $\Lambda$.
The following result has no restrictions on $d$ and hence can serve as a substitute for \eqref{edge2side} for sparser regular digraphs (recall that \Cref{thm_main} is most useful when $\min(d,n-d)=\omega(\log n)$ -- see \Cref{rmk_dgrows}).

\begin{theorem}[Concentration of edge counts, permutation model]		\label{thm_edgeperm}
Let $n,d\ge 1$, and $A,B\subset[n]$. 
For any $\tau\ge 0$ we have
\begin{equation}
\pr\Big( \big| e_{\Lambda}(A,B) - \mu \big| \ge \tau \mu \Big) \le 
2\expo{ - \frac{\tau^2\mu}{2+\tau}}
\end{equation}
where $\mu=\pe|A||B|$ as before, and $e_\Lambda(A,B)$ is the number of directed edges from $A$ to $B$, counting multiplicity. 
\end{theorem}

The above theorem is considerably easier to establish than part (2) of \Cref{thm_main} -- it turns out that the independence between the $d$ factors allows one to proceed with the method of switchings without needing a priori bounds on codegrees. 
We will hence prove \Cref{thm_edgeperm} as a warmup in \Cref{sec_com}.

Sharper bounds for larger deviations (i.e. when $\tau$ is large) can be proved by directly estimating $\pr(e(A,B)=t)$ for all $t\in \N$, leading to an estimate on the moment generating function $m(\theta)=\e \exp(\theta\, e(A,B))$. 
This was the route taken in \cite{DJPP} to prove a certain discrepancy property for the permutation model.

\subsection{Extension to general bipartite regular graphs}		\label{sec_extend}

\Cref{thm_main} above easily extends to the following more general setting. 
For $m,n\ge 1$ and $d\in [n],$ $d'\in [m]$, draw $\Gamma = (U,V,E)$ uniformly from the set of bipartite graphs on parts $U$, $V$ with $|U|=m$, $|V|=n$ and edge set $E\subset U\times V$, with the constraint that each $i\in U$ has degree $d$ and each $j\in V$ has degree $d'$.
Since the total number of edges is
$$md= nd'$$
we denote
\begin{equation}
\theta= \frac{m}{n} = \frac{d'}{d}.
\end{equation}
The random regular digraph considered above corresponds to the case $\theta=1$. 
As before, we identify $U$ with $[m]$ and $V$ with $[n]$, denote by $p=d/n=d'/m$ the edge density of $\Gamma$, $\hat{d}:= \min(d,n-d)$, and $\hat{\mu}(A,B):= p\min\{|A||B|, (m-|A|) (n-|B|)\}$.

The following result is proved by the same lines as \Cref{thm_main}, only with slightly more burdensome notation.

\begin{theorem}[Extension to bipartite graphs]
\label{thm_bip}

For $\eta\ge 0$ define the event
\begin{equation}
\good^{\co}(\eta) = \bigg\{ \forall \set{i_1,i_2}\subset[m], \;  \left|\harp{\co}_\Gamma(i_1,i_2) - p^2n\right| \le \eta p(1-p)n \bigg\}.
\end{equation}
We have
\begin{enumerate}
\item (Uniform control on codegrees) For any $\eta\ge 0$, $\good^{\co}(\eta)$ holds except with probability
\begin{equation}		
O\bigg( m^2\hat{d}^2 \expo{-\frac{c\eta n^2}{m}} \bigg) 
+ O\bigg( m^2 \expo{ -c\eta \min\Big\{  \hat{d}, \eta n\Big\} } \bigg).
\end{equation}
In particular, $\good^{\co}(\eta)$ holds a.a.s.\ in the limit $m,n\rightarrow \infty$ as long as 
$$\eta\ge C\max\left\{\frac{m\log (m\hat{d})}{n^2}, \frac{ \log m}{\hat{d}}, \sqrt{\frac{\log m}{n}}\right\}$$
for some $C>0$ sufficiently large. 
\item (Concentration of edge counts) For any $A\subset[m]$, $B\subset[n]$ and any $\tau\ge 0$, 
if  $\eta \le \min\left(\frac{1}{4},\frac{\tau}{8}\right)$, we have
\begin{equation}		
\pro{  \Big\{ \disc_\Gamma(A,B) \ge \tau \hat{\mu}(A,B)   \Big\}  \wedge \good^{\co}\big(\eta\big) } 
\le 2 \expo{-\frac{\tau^2}{C_1+C_2\tau}\hat{\mu}(A,B)}.
\end{equation}

\end{enumerate}

\end{theorem}

\vspace{.3cm}

The rest of the paper is organized as follows. 
In \Cref{sec_com} we introduce and motivate Chatterjee's method of exchangeable pairs in the context of two random digraph models that are simpler to analyze than the uniform random regular digraph, namely the Erd\H os--R\'enyi model and the permutation model (as defined in \eqref{Mper}).
The proof of \Cref{thm_edgeperm} is given in \Cref{sec_perm}.
In \Cref{sec_couplings} we construct the switching and reflection couplings, which will be used to create exchangeable pairs of random regular digraphs.
In \Cref{sec_codeg1} we use the reflection coupling to prove an upper tail bound for the codegree of a fixed pair of vertices. 
For technical reasons the proof of the lower tail bounds requires more care, in particular using the control on the upper tail as input -- this is carried out in \Cref{sec_codeg2}, completing the proof of part (1) of \Cref{thm_main}.
The tail bounds for edge discrepancy in part (2) of \Cref{thm_main} are proved using the simple switching coupling in \Cref{sec_edge}.

\subsection{Notation}			\label{sec_notation}

We make use of the following asymptotic notation with respect to the limit $n=|V|\rightarrow \infty$.
$f\ll g$, $g\gg f$, $f=O(g)$, and $g=\Omega(f)$ are all synonymous to the statement that $|f|\le Cg$ for all $n\ge C$ for some absolute constant $C$. $f\asymp g$ and $f=\Theta(g)$ mean $f\ll g$ and $f\gg g$.
$f=o(g)$ and $g=\omega(f)$ mean that $f/g\rightarrow 0$ as $n$ tends to infinity. 
For a parameter $\alpha\in \R$, 
$f\ll_\alpha g$, $f=O_\alpha(g)$ etc. mean that $|f|\le C_\alpha g$ for all $n\ge C_\alpha$, with $C_\alpha$ a constant depending only on $\alpha$.
$C,c, c', c_1$, etc. denote absolute constants whose value may change from line to line.

Events will be denoted by the letters $\event, \bad,$ and $\good$, where the latter two denote ``bad" and ``good" events, respectively. 
Their meaning may vary from proof to proof, but will remain fixed for the duration of each proof. $\un_{\event}$ denotes the indicator random variable corresponding to the event $\event$, and for a statement $S$, $\un(S)=\un_{\set{\mbox{$S$ holds}}}$.
$\e_X$ and $\pr_X$ denote expectation and probability, respectively, conditional on all random variables but $X$.
We say that an event $\event$ depending on $n$ holds \emph{asymptotically almost surely} if $\pro{\event^c}=o(1)$.

It will be convenient to express codegrees and edge counts in terms of the adjacency matrix associated to $\Gamma$, which we denote by $M$. 
We also denote by $\mM_{n,d}$ the set of all adjacency matrices associated to the elements of $\mD_{n,d}$ (alternatively, this is the set of $n\times n$ matrices with entries in $\set{0,1}$, subject to the constraint that each row and column contains exactly $d$ 1s).
Hence $M$ is a uniform random element of $\mM_{n,d}$, and we refer to it as an \emph{rrd matrix} (for ``random regular digraph").

We identify $V$ with $[n]$ and index the rows and columns of $M$ by $i$ and $j$, respectively. 
By abuse of notation we refer to $i, j$ as ``vertices". 
Given ordered tuples of row and column indices $(i_1,\dots, i_a)$ and $(j_1,\dots, j_b)$, we denote by $M_{(i_1,\dots, i_a)\times(j_1,\dots, j_b)}$ the $a\times b$ matrix with $(k,l)$ entry equal to the $(i_k,j_l)$ entry of $M$. 
(Note for instance that the sequence $(i_1,\dots,i_a)$ need not be increasing.)

For $i\in [n]$, let
\begin{equation}
\mN_M(i) = \set{j\in [n]: M(i,j)=1}
\end{equation}
so that $\mN_M(i)$ and $\mN_{M^\tran}(i)$ are the out- and in-neighborhoods of the vertex $i$, respectively. 
For the neighborhood of a pair of distinct vertices $i_1,i_2\in [n]$, denote the set of common out-neighbors by
\begin{align}
\Co_M(i_1,i_2) &=  \mN_M(i_1)\cap \mN_M(i_2) \\
&= \set{j\in [n]: M_{(i_1,i_2)\times j}={1\choose 1}}
\end{align}
and denote also
\begin{align}
\Ex_M(i_1,i_2) &= \mN_M(i_1)\setminus \mN_M(i_2)\\
&= \set{j\in [n]: M_{(i_1,i_2)\times j}={1\choose 0}}
\end{align}
so that
\begin{equation}
\Ex_M(i_2,i_1) = \set{j\in [n]: M_{(i_1,i_2)\times j}={0\choose 1}}.
\end{equation}
We write $\co_M(i_1,i_2)$ and $\ex_M(i_1,i_2)$ for the cardinality of these sets, so that in our previous notation
\begin{align*}
\harp{\co}_\Gamma(i_1,i_2)&=  \co_M(i_1,i_2) \\
 \garp{\co}_\Gamma(j_1,j_2) &=\co_{M^\tran}(j_1,j_2)=\left|\set{i\in [n]: M_{i\times(j_1,j_2)}=(1\quad 1)}\right|.
\end{align*}

We note the following identities. 
From the constraints $\sum_{j=1}^nM(i_1,j)=d$ and $\sum_{j=1}^nM(i_2,j)=d$ we have
\begin{align}
\ex_M(i_1,i_2) & =  d-\co_M(i_1,i_2)	= \ex_M(i_2,i_1) \label{id23}.
\end{align}
Finally we have that
\begin{align}	\label{id4}
\left|\set{j\in [n]: M_{(i_1,i_2)\times j}={0\choose 0}}\right| = n-2d+\co_M(i_1,i_2).
\end{align}
We will also write $e_M(A,B)$ instead of $e_\Gamma(A,B)$.

\section{Concentration of measure and exchangeable pairs}		\label{sec_com}

In this section we prove analogues of the bounds in \Cref{thm_main} for two digraph models possessing more independence than the uniform $d$-regular digraph: the Erd\H os--R\'enyi model, in which all edges are independent, and the permutation model, as defined in \eqref{Mper}. 
The proofs for the former model illustrate the application of concentration of measure tools, and are completely standard.
Their use of Chernoff-type bounds (namely Bernstein's inequality), which are unavailable for random regular graphs, motivate the method of exchangeable pairs (Chatterjee's \Cref{thm_chatterjee}) as a substitute.
We prove \Cref{thm_edgeperm} for the permutation model in \Cref{sec_perm} as a simple illustration of the method.
The reader who is primarily interested in getting a feel for applying the method to combinatorial problems may prefer to read the proof of \Cref{thm_edgeperm} to the more technical proof of part (2) of \Cref{thm_main} in \Cref{sec_edge}.

\subsection{The Erd\H os--R\'enyi model} 		\label{sec_er}

Let $D=(V,E)$ be drawn from the distribution $D(n,p)$ over digraphs on $n$ vertices, where each directed edge is included independently with probability $p$. 

\begin{proposition}[Uniform control of codegrees and edge counts, Erd\H os--R\'enyi case] 
\label{prop_er}
\mbox{ }
\begin{enumerate}[(i)]
\item For any $\eps> 0$, except with probability $O\Big(n^2\expo{-\frac{c\eps^2}{1+\eps} p^2n}\Big)$ we have that
for all $i_1,i_2\in [n]$ distinct,
$$ \Big|\co_D(i_1,i_2) - p^2n \Big| \le \eps  p^2n.$$
\item For $\eps\in (0,1)$, let
\begin{equation}
\mF(\eps)= \set{(A,B):\; A, B\subset[n],\; \min\big(|A|,|B|\big) \ge \frac{C_0\log n}{\eps^2p} }
\end{equation}
where $C_0>0$ is a sufficiently large absolute constant. 
For any $\eps\in (0,1)$, with probability $1-O\Big(\expo{ -c\frac{\log^2n}{\eps^2p}}\Big)$ we have that 
for all $(A,B)\in \mF(\eps)$, 
\begin{equation}
 \Big|e_D(A,B) - p|A||B| \Big| \le \eps p |A||B|.
\end{equation}

\end{enumerate}

\end{proposition}

\begin{proof}

For fixed vertices $i_1,i_2\in V$ and subsets $A,B\subset V$, the statistics $\co_D(i_1,i_2)$ and $e_D(A,B)$ can be expressed as sums of iid indicator variables:
\begin{align}
\co_D(i_1,i_2) &= \sum_{j=1}^n \un((i_1,j)\in E)\un((i_2,j)\in E), \label{co_er}\\
e_D(A,B) &= \sum_{i\in A, j\in B} \un((i,j)\in E).	\label{e_er}
\end{align}
It follows that $\e \co_D(i_1,i_2)= p^2n$ and $\e e_D(A,B)= p|A||B|$. 
Furthermore, by Bernstein's inequality we have that for any $\eps\ge 0$, 
\begin{equation}	\label{erco}
\pr\Big[ \big|\co_D(i_1,i_2) - p^2n \big| \ge \eps p^2n\Big] \le 2 \expo{-\frac{c\eps^2}{1+\eps} p^2n}
\end{equation}
and
\begin{equation}		\label{eredge}
\pr\Big[ \big|e_D(A,B) - p|A||B| \big| \ge \eps p|A||B|\Big]  \le 2\expo{ - \frac{c\eps^2}{1+\eps} p|A||B|}
\end{equation}
for some absolute constant $c>0$. 
From \eqref{erco} and a union bound we obtain uniform control of codegrees off a small event:
\begin{equation}	\label{ercocontrol}
\pro{ \exists \mbox{ distinct } i_1,i_2 \in [n] : \Big|\co_D(i_1,i_2) - p^2n \Big| \ge \eps  p^2n } \ll n^2 \expo{-\frac{c\eps^2}{1+\eps} p^2n}
\end{equation}
which establishes (i). 

The proof of (ii) follows the same lines as in the proof of \Cref{cor_discrep} (in particular the part establishing \eqref{sts}), using the bounds \eqref{eredge} in place of \eqref{edge2side}.
\end{proof}

\subsection{Chatterjee's method of exchangeable pairs}		\label{sec_chatterjee}

The main challenge for proving analogous results for $d$-regular digraphs is that the entries of $M$ are all dependent on one another, and so we cannot apply off-the-shelf concentration of measure tools like Bernstein's inequality. 
The method of exchangeable pairs, as developed by Stein for normal-approximation \cite{Stein} and by Chatterjee for concentration of measure \cite{Chatterjee}, provides a convenient framework for analyzing dependent structures possessing measure preserving actions of a ``local" nature. 
We will use this to obtain bounds of the form \eqref{erco} and \eqref{eredge} for the random regular digraph $\Gamma$.

Recall that a pair of $\mM$-valued random variables $(M_1,M_2)$ is exchangeable if
$$(M_1,M_2)\eqd (M_2,M_1).$$
In particular we have $M_1\eqd M_2$. 
We will consider exchangeable pairs $(M,\Phi(M))$ formed by the application of a transformation $\Phi:\mM\rightarrow \mM$ with certain properties.
Roughly speaking, the method derives properties of a statistic $f(M)$, such as concentration or approximate normality, by analyzing the change in $f(M)$ under the application of $\Phi$. 

An example of a ``local" measure-preserving operation for a sequence of independent variables is to resample one of the variables independently of all others. 
For $d$-regular graphs, there are switching operations (described in \Cref{sec_switching}).

The following is a version of Theorem 1.5 from \cite{Chatterjee} suitable for our purposes:

\begin{theorem}[Chatterjee \cite{Chatterjee}]			\label{thm_chatterjee}
Let $\mM$ be a separable metric space, and suppose $(M,\tM)$ is an exchangeable pair of $\mM$-valued random variables, i.e.\
$$(M,\tM)\eqd (\tM,M).$$ 
Suppose $f:\mM\rightarrow \mathbf{R}$ and $F:\mM\times\mM\rightarrow \mathbf{R}$ are square-integrable functions such that $F(M,\tM)=-F(\tM,M)$ $a.s.$ and $\e(F(M,\tM)|M)=f(M)$ $a.s.$. Assume
\begin{equation}\label{finitemoment}
\e\left[e^{\theta f(M)}\big|F(M,\tM)\big|\right]<\infty
\end{equation}
for all $\theta\in \R$. Let
$$
v_f(M):=\frac{1}{2}\e\left[\big|\big(f(M)-f(\tM)\big)F(M,\tM)\big|\,\Big|\,M\right].
$$
If there are non-negative constants $K_1,K_2$ such that $v_f(M)\le K_1+K_2f(M)$ $a.s.$, then for any $t\ge 0$,
\begin{equation}
\pr(f(M)\ge t)\le \expo{-\frac{t^2}{2(K_1+K_2t)}}, \quad \pr(f(M)\le -t)\le \expo{-\frac{t^2}{2K_1}}.
\end{equation}

\end{theorem}

\begin{remark}
The qualitative integrability conditions on $f$ and $F$ will be satisfied automatically in our applications as we will only consider bounded (depending on $n$) functions on a finite set.
\end{remark}

The quantity $v_f(M)$ is referred to by Chatterjee as a ``stochastic measure of the variance of $f(M)$", and one can view a bound of the form
$$v_f(M)\le K_1+K_2f(M)$$
as a generalization of the ``Lipschitz" conditions assumed in other commonly used concentration bounds such as McDiarmid's inequality \cite{McDiarmid}. 
We point the reader to \cite{Chatterjee} for further discussion of \Cref{thm_chatterjee} and its relation to other concentration inequalities.

\subsection{The permutation model: Proof of \Cref{thm_edgeperm}}		\label{sec_perm}

In this section we illustrate how one applies \Cref{thm_chatterjee} by proving the edge discrepancy bounds of \Cref{thm_edgeperm} for the permutation model $\Lambda$. 
The proof is a cartoon of the proof of the analogous bound from \Cref{thm_main} for the uniform model, given in \Cref{sec_edge}.
Various technical issues that must be addressed for the case of the uniform model are absent here; in particular, the independence between the permutation matrices allows us to proceed without any a priori control on codegrees.

We recall from \Cref{sec_results} that the permutation model $d$-regular directed multigraph $\Lambda$ has adjacency matrix given by
$$M_{\Lambda}= P_1+\cdots + P_d$$
where $P_1,\dots,P_d$ are iid uniform $n\times n$ permutation matrices. 
We may hence view the statistics $e_{\Lambda}(A,B)$ as functions of a uniform random element $\pi=(\pi_1,\dots, \pi_d)$ of $\sym(n)^d$, where $\sym(n)$ denotes the symmetric group over $[n]$. 
For $\sigma\in \sym(n)$ and $A,B\subset[n]$, denote
\begin{equation}
e_\sigma(A,B) = \big| \big\{ i\in A: \sigma(i)\in B \big\} \big|
\end{equation}
and for $\pi=(\pi_1,\dots,\pi_d)\in \sym(n)^d$ we set
\begin{equation}
e_\pi(A,B) = \sum_{k=1}^d e_{\pi_k}(A,B).
\end{equation}
If $\pi$ is a uniform random element of $\sym(n)^d$ we hence have
$$e_{\Lambda}(A,B) \eqd e_\pi(A,B).$$
\Cref{thm_edgeperm} is then a consequence of the following

\begin{proposition}		
\label{prop_perm}
If $\pi=(\pi_1,\dots,\pi_d)$ is a uniform random element of $\sym(n)^d$ and $A,B$ are fixed subsets of $[n]$, we have that for any $\tau\ge 0$, 
\begin{equation}
\pr\left\{e_\pi(A,B) \ge (1+\tau)\frac{d}n|A||B| \right\} \le \expo{-\frac{\tau^2}{2+\tau}\frac{d}{n}|A||B|}
\end{equation}
and
\begin{equation}
\pr\left\{ e_\pi(A,B) \le (1-\tau)\frac{d}n|A||B| \right\} \le \expo{-\frac{\tau^2}{2}\frac{d}{n}|A||B|}.
\end{equation}

\end{proposition}

The proof is similar to the proof of Proposition 1.1 in \cite{Chatterjee}, which was concerned with a more general statistic but for the case of $d=1$. 
Here and in the remainder of the paper we will make use of the following 
\begin{observation}[Exchangeable pair from an involution]		\label{obs_invo}
Let $\mM$ be a finite set, and suppose $\Phi:\mM\rightarrow \mM$ is an involution. 
Let $M$ be a uniform random element of $\mM$, and set $\tM=\Phi(M)$. 
Then $(M,\tM)$ is an exchangeable pair of uniformly distributed elements of $\mM$.

\end{observation}

\begin{proof}
Since $M$ is uniform and $\Phi$ is a permutation we have $\Phi(M)\eqd M$, and so
$$(M,\tM) = (M,\Phi(M)) \eqd (\Phi(M),\Phi^2(M)) = (\tM,M).$$
\end{proof}

\begin{proofof}{\Cref{prop_perm}}
Define the anti-symmetric function $F: \sym(n)^d\times \sym(n)^d\rightarrow \R$ by
$$F(\pi,\pi') = K\big[ e_{\pi}(A,B) - e_{\pi'}(A,B) \big]$$
where $K$ is a normalizing constant. 
With foresight we take
\begin{equation}		\label{mK}
K=\frac{d}{n}a(n-a)
\end{equation}
where we denote $|A|=a$, $|B|=b$. 

We construct an exchangeable pair $(\pi,\tpi)$ of uniform random elements of $\sym(n)^d$ as follows. We draw the following random variables, uniformly at random from their respective ranges:
\begin{itemize}
\item $\pi= (\pi_1,\dots, \pi_d) \in \sym(n)^d$,
\item $J\in [d]$,
\item $I_1\in A$
\item $I_2\in [n]\setminus A$
\end{itemize}
with $\pi, J, I_1,I_2$ jointly independent. 
We form $\tpi$ by replacing $\pi_J$ with $\uptau_{\set{I_1,I_2}}\circ \pi_J$, where $\uptau_{\set{i_1,i_2}}$ denotes the transposition of $i_1,i_2\in [n]$; $(\pi_k)_{k\ne J}$ are left unchanged. 
$(\pi,\tpi)$ is an exchangeable pair by \Cref{obs_invo}. 
We have
\begin{align}
e_{\pi}(A,B) - e_{\tpi}(A,B)  &= e_{\pi_J}(A,B) - e_{\uptau_{\set{I_1,I_2}}\circ \pi_J}(A,B) \\
&= \un(\pi_J(I_1)\in B) \un(\pi_J(I_2)\notin B) - \un(\pi_J(I_1)\notin B)\un(\pi_J(I_2)\in B)
\end{align}
and so
\begin{align}
f(\pi) &:= \e\big[ F(\pi,\tpi) \big| \pi \big] \notag\\
&= K \Big[ \pr\Big\{ {\pi_J}(I_1)\in B, \; {\pi_J}(I_2) \notin B \big| \pi\Big\} - \pr\Big\{{\pi_J}(I_1)\notin B, \; {\pi_J}(I_2) \in B \big| \pi \Big\} \Big] \label{pmp}\\
&= K \e_J\left[ \frac{e_{\pi_J}(A,B)}{a} \frac{e_{\pi_J}(A^c,B^c)}{n-a} - \frac{e_{\pi_J}(A,B^c)}{a} \frac{e_{\pi_J}(A^c,B)}{n-a} \right] \notag\\
&= \frac{K}{a(n-a)}\e_J\Big[ e_{\pi_J}(A,B)\big( n-a-b + e_{\pi_J}(A,B)\big) - \big( a - e_{\pi_J}(A,B) \big) \big( b- e_{\pi_J}(A,B)\big) \Big] \notag\\
&= \frac{Kn}{a(n-a)}\e_Je_{\pi_J}(A,B) - \frac{dab}n \\
&= e_\pi(A,B) - \frac{dab}{n}
\end{align}
where in the last line we applied \eqref{mK}.

It remains to bound the quantity $v_f(\pi)$ from \Cref{thm_chatterjee}.
We have
\begin{align*}
(f(\pi)-f(\tpi))^2 &= (e_\pi(A,B) - e_{\tpi}(A,B))^2\\
&= \un({\pi_J}(I_1)\in B) \un({\pi_J}(I_2)\notin B) + \un({\pi_J}(I_1)\notin B)\un({\pi_J}(I_2)\in B)
\end{align*}
so
\begin{align*}
v_f(\pi) &:= \frac{1}{2} \e\big[ |f(\pi)-f(\tpi)||F(\pi,\tpi)| \big| \pi\big] \\
&=  \frac{K}{2} \e \big[ \big( f(\pi)-f(\tpi)\big)^2 \big| \pi \big] \\
&= \frac{K}2 \Big[ \pr\Big\{ {\pi_J}(I_1)\in B, \; {\pi_J}(I_2) \notin B \big| \pi\Big\} + \pr\Big\{{\pi_J}(I_1)\notin B, \; {\pi_J}(I_2) \in B \big| \pi \Big\} \Big] \\
&= \frac12 f(\pi) + K\pr\Big\{{\pi_J}(I_1)\notin B, \; {\pi_J}(I_2) \in B \big| \pi \Big\} \\
&= \frac12f(\pi) + K\e_J\frac{ \big( a - e_{\pi_J}(A,B) \big) \big( b- e_{\pi_J}(A,B)\big)}{a(n-a)}\\
&\le \frac{1}{2}f(\pi) + \frac{d}{n}ab
\end{align*}
where in the fourth line we applied \eqref{pmp}.

By \Cref{thm_chatterjee} we conclude that for any $t\ge0$, 
\begin{align}
\pr\bigg\{ e_\pi(A,B) - \frac{d}{n}ab \ge t \bigg\} &\le \expo{ -\frac{t^2}{2\frac{d}{n}ab + t}} \\
\pr\bigg\{ e_\pi(A,B) - \frac{d}{n}ab \le -t \bigg\} &\le \expo{ -\frac{t^2}{2\frac{d}{n}ab }}.
\end{align}
Setting $t=\tau \frac{d}{n}ab$ completes the proof. 
\end{proofof}

\section{Exchangeable pairs constructions}			\label{sec_couplings}

In this section we define two involutions on $\mM_{n,d}$ -- \emph{simple switchings} and \emph{reflections} -- which we use to create exchangeable pairs $(M,\tM)$ of rrd matrices via \Cref{obs_invo}.

\subsection{Simple switching}		\label{sec_switching}

Below we set up our notation for switchings on a digraph in terms of the adjacency matrix $M$.

\begin{definition}[Simple switching]	\label{def_switching}
For $M\in \mM_{n,d}$ and $i_1,i_2,j_1,j_2\in [n]$, we say that the $2\times2$ minor $M_{(i_1,i_2)\times(j_1,j_2)}$ is \emph{switchable} if it is equal to either
\begin{equation}	\label{eyedef}
\eye=\mat[1]{0}{0}{1}\quad \mbox{or}\quad \jay=\mat[0]{1}{1}{0}. 
\end{equation}
By \emph{perform a switching at $(i_1,i_2)\times(j_1,j_2)$ on $M$} we mean to replace the minor $M_{(i_1,i_2)\times(j_1,j_2)}$ with $\jay$ if it is $\eye$ and $\eye$ if it is $\jay$, and to leave $M$ unchanged if this minor is not switchable. 
\end{definition}

In the associated digraph $\Gamma$, the switching operation changes between the following edge configurations at vertices $i_1,i_2,j_1,j_2$:
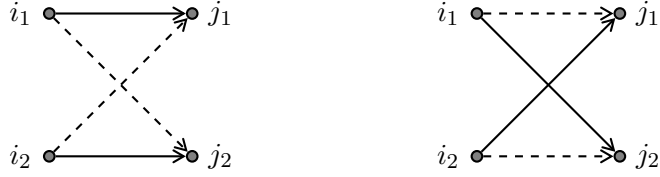
\begin{center}

\eyepic \hspace{2.2cm} \jaypic

\end{center}
where we use solid arrows to depict directed edges, and dashed arrows to indicate places where there is no edge (i.e.\ ``non-edges").

\begin{lemma}[Switching coupling] 	\label{lem_switching}
For $i_1,i_2,j_1,j_2\in [n]$, let $\Phi_{(i_1,i_2)\times(j_1,j_2)}:\mM_{n,d}\rightarrow \mM_{n,d}$ denote the map which performs a simple switching at the minor $(i_1,i_2)\times(j_1,j_2)$. 
If $M$ is an rrd matrix (i.e.\ a uniform random element of $\mM_{n,d}$) and $I_1,I_2,J_1,J_2\in [n]$ are random (or deterministic) indices independent of $M$, then setting
\begin{equation}
\tM:= \Phi_{(I_1,I_2)\times(J_1,J_2)}(M)
\end{equation}
we have that $(M,\tM)$ is an exchangeable pair of rrd matrices.

\end{lemma}

\begin{proof}
We may condition on $I_1,I_2,J_1,J_2$. 
Note that the map $\Phi_{(I_1,I_2)\times(J_1,J_2)}$ is an involution on $\mM_{n,d}$. 
The result now follows from \Cref{obs_invo}.
\end{proof}

\subsection{Reflection}			\label{sec_reflection}

In order to prove that the random variables $\co_M(i_1,i_2)$ are concentrated we will need a different 
operation on random regular digraphs of switching-type which we call ``reflection". 
We pause to give some motivation and intuition for the rigorous definition below.

Suppose first that we only want to prove an upper tail bound on $\co_M(1,2)$. 
Hence, we want to show it is unlikely that for most $j\in [n]$ we have
$$M_{(1,2)\times j} = {1\choose 1} \quad\mbox{or}\quad {0\choose 0}$$
i.e., that the first two rows of $M$ are nearly parallel.
The idea is to show that for a pair of column indices $j_1,j_2\in [n]$, the event that 
\begin{equation}		\label{K2}
M_{(1,2)\times (j_1,j_2)}=\mat[1]{0}{1}{0}
\end{equation}
is roughly just as likely as the event that
\begin{equation}		\label{I2}
M_{(1,2)\times (j_1,j_2)}=\mat[1]{0}{0}{1}.
\end{equation}
We will do this by defining a ``reflection" operation which switches the $(1,2)\times(j_1,j_2)$ minor between these two outcomes. 
If we can perform reflections independently at random at several disjoint pairs of column indices, we can then deduce from Hoeffding's inequality that with high probability there are many columns $j$ for which
$$M_{(1,2)\times j}={1\choose 0}\quad\mbox{or} \quad{0\choose 1}$$
as desired. 
While this approach can be made precise, we can do much better by instead using \Cref{thm_chatterjee}, which gives upper and lower tail estimates for $\co_M(1,2)$ around its mean. 

While it is possible to alternate between the minors \eqref{K2} and \eqref{I2} using simple switchings involving entries from a third row, it turns out that when one tries to apply \Cref{thm_chatterjee} with this coupling some control on the quantities $\co_{M^\tran}(j_1,j_2)$ is needed, so that such an approach is circular. 

The reflection involution is most natural to state in terms of a \emph{walk} $w_{(j_1,j_2)}:[n]\rightarrow \Z$ associated to an ordered pair of columns $X_{j_1},X_{j_2}$ of $M$. 
For $(j_1,j_2)\in [n]^2$ we define
\begin{equation}		\label{def_walk}
w_{(j_1,j_2)}(i)=\sum_{k=1}^i \un\big(M_{k\times (j_1, j_2)}=(1\; 0)\big)- \un\big(M_{k\times(j_1, j_2)}=(0\; 1)\big).
\end{equation}
If we think of $w_{(j_1,j_2)}$ as giving the position of a walker on $\Z$, the walker starts at 0 and, reading down the pair of columns $(X_{j_1},X_{j_2})$ of $M$, takes a step in the $+$ direction each time it sees a row equal to $(1\; 0)$, a step in the $-$ direction each time it reads $(0\; 1)$, and does not move otherwise. 
By $d$-regularity, the walker takes an even number of steps, half to the left and half to the right, ending its walk at 0.
The number of steps is between $0$ and $2d$; in the former case $X_{j_1}$ and $X_{j_2}$ are parallel, and in the latter case they are orthogonal.

\begin{definition}[Reflecting pair]
\label{def_reflection}
With $w_{(j_1,j_2)}$ as in \eqref{def_walk}, we say that an ordered pair of column indices $(j_1,j_2)\in [n]^2$ is \emph{reflecting} for $M\in \mM_{n,d}$ if 
\begin{enumerate}
\item $w_{(j_1,j_2)}(1)=+1$, 
\item $w_{(j_1,j_2)}(2)\ne +1$, and

\item there exists $i\in [3,n]$ such that $w_{(j_1,j_2)}(i) = +1$

\end{enumerate}
that is, if the walker moves to $+1$ on the first step, leaves $+1$ on the second step, and returns again to $+1$ at some later time.

\end{definition}

Conditions (1) and (2) above assert that the minor $M_{(1,2)\times(j_1,j_2)}$ is either 
$$\kay:= \mat[1]{0}{1}{0} \quad \mbox{or}\quad \eye= \mat[1]{0}{0}{1} .$$
We pause to note that condition (3) usually holds if (1) and (2) hold. 
Indeed,
note that if $w_{(j_1,j_2)}(2)=+2$ then condition (3) follows automatically from (1) and (2) since the walk must pass through $+1$ on its way back to 0. Hence, any pair $(j_1,j_2)$ of column indices such that $M_{(1,2)\times (j_1,j_2)}=\kay$ is reflecting. 

On the other hand, note that a non-reflecting pair $(j_1,j_2)$ for which $M_{(1,2)\times (j_1,j_2)}=\eye$ corresponds to a walk $w_{(j_1,j_2)}$ that reaches $+1$ on the first step, then turns back and never returns to $+1$.
Non-reflecting pairs satisfying (1) and (2) but not (3) hence correspond to walks that do not cross the line $w=0$ after the second step, 
so we can bound the probability of this happening by a standard enumerative argument involving Catalan numbers. 
This is carried out in the proof of \Cref{lem_gbound}.
Consequently, one may think of reflecting pairs as essentially being those $(j_1,j_2)\in [n]^2$ such that $M_{(1,2)\times(j_1,j_2)}= \kay$ or $\eye$.

If $(j_1,j_2)$ is reflecting for $M$, denote by
$$i^*(j_1,j_2):=\min\set{i\in [3,n]:w_{(j_1,j_2)}(i)=+1}$$
the first return time to +1.

\begin{lemma}[Reflection coupling] 		\label{lem_reflection}
For $(j_1,j_2)\in [n]^2$, let $\;\Psi_{(j_1,j_2)}:\mM_{n,d}\rightarrow \mM_{n,d}$ denote the map which replaces the minor $M_{[2,i^*]\times(j_1,j_2)}$ with the ``reflected" minor $M_{[2,i^*]\times(j_2,j_1)}$ if $(j_1,j_2)$ is reflecting, and leaves $M$ unchanged otherwise. 
If $M$ is an rrd matrix and $J_1,J_2\in [n]$ are random column indices independent of $M$, then setting
\begin{equation}
\tM:= \Psi_{(J_1,J_2)}(M)
\end{equation}
we have that $(M,\tM)$ is an exchangeable pair of rrd matrices. 

\end{lemma}

\begin{proof}
By conditioning on $J_1,J_2$, from \Cref{obs_invo} it suffices to show that $\Psi_{(j_1,j_2)}$ is an involution on $\mM_{n,d}$ for $j_1,j_2\in [n]$ fixed.  

$\Psi_{(j_1,j_2)}$ acts trivially if $j_1=j_2$, so we may fix $j_1,j_2\in [n]$ distinct. 
We can now divide $\mM_{n,d}$ into three classes:
\begin{enumerate}
\item $\mM^0_{(j_1,j_2)}$, the set of $M$ such that $(j_1,j_2)$ is not reflecting for $M$. 
\item $\mM^+_{(j_1,j_2)}$, the set of $M$ such that $(j_1,j_2)$ is reflecting for $M$ and $w_{(j_1,j_2)}(2)=+2$. 
\item $\mM^-_{(j_1,j_2)}$, the set of $M$ such that $(j_1,j_2)$ is reflecting for $M$ and $w_{(j_1,j_2)}(2)=0$. 
\end{enumerate}
We dispense with the subscripts $(j_1,j_2)$ for the remainder of the proof.

$\Psi$ acts trivially on $\mM^0$. 
We will show that $\Psi$ is a bijection between $\mM^+$ with $\mM^-$ with $\Psi^2 = \mbox{Id}$. 

We define a pairing $\mP$ of the elements of $\mM^+$ with those of $\mM^-$ (in particular, these sets have the same cardinality). 
For $(M^+,M^-)\in \mM^+\times \mM^-$, let $w^+$ and $w^-$ denote the associated walks for the columns $(j_1,j_2)$. 
We say that $(M^+,M^-)$ is in $\mP$ if the first return time $i^*$ of the walks $w^+, w^-$ to $+1$ is the same, and if the walk $w^+$ is obtained from $w^-$ by reflecting the portion of the trajectory 
of $w^-(i)$ with $i\in [2,i^*]$ across the line $w=+1$.
We conclude the proof by noting that $\Psi$ sends each $M\in \mM^+\cup \mM^-$ to its mate in $\mP$. 
\end{proof}

\begin{remark}
The bijection $\Psi_{(j_1,j_2)}$ above is an application of the well-known \emph{reflection principle} from the theory of random walks -- see for instance \cite[Chapter III]{Feller:vol1}.
\end{remark}

\section{The upper tail for codegrees}		\label{sec_codeg1}

Our aim in this section is to prove the following

\begin{proposition}[Upper tail for codegree]		\label{prop_codeg}
For any $\eps\ge 0$ and any distinct $i_1,i_2\in [n]$, 
\begin{equation}		\label{cotailup}
\pr\Big\{\co_M(i_1,i_2)-\pe^2 n\ge \eps\hat{\pe}^2n\Big\} \le \expo{ -\frac{\eps^2}{4+2\eps}\hat{\pe}^2n}.
\end{equation}
\mbox{ }

\end{proposition}
\begin{remark}[Comparison to the \er case] 
Up to constants in the exponential, this matches the upper tail for the \er digraph given in \eqref{erco}.
\end{remark}

\begin{remark}
As a corollary one may obtain some control on edge discrepancy by applying the above proposition (with a union bound over pairs of vertices) with \Cref{lem_pseudo}. This will only be effective when $d=\omega(\sqrt{n})$ and for pairs of sets $A,B$ with $\max(|A|,|B|)\gg n$, and is hence inferior to \Cref{cor_discrep}. 
\end{remark}

\begin{proof}
We will apply \Cref{thm_chatterjee} and the reflection coupling of \Cref{lem_reflection}.

We first note the trivial deterministic bounds
\begin{equation}		\label{whatid}
d\ge \co_M(i_1,i_2) \ge \max(0, 2d-n).
\end{equation}
The lower bound is equivalent to 
\begin{equation}	\label{exub}
\ex_M(i_1,i_2)\le \min(d,n-d) = \hat{d}
\end{equation}
which can be seen from the obvious bound 
$$\ex_M(i_1,i_2)= \big| \mN_M(i_1)\setminus \mN_M(i_2)\big| \le \big|\mN_M(i_1)\big|=d$$ 
and the fact that $\ex_M(i_1,i_2)=\ex_{M'}(i_1,i_2)\le n-d$ (where $M'$ is the adjacency matrix of the complementary digraph $\Gamma'$).

Since the rows of $M$ are exchangeable we may take $(i_1,i_2)=(1,2)$. 
Let us abbreviate 
$$\co(M):= \co_M(1,2).$$
We construct a coupled pair $(M,\tM)$ of rrd matrices as follows: letting $M$ be an rrd matrix and $J_1,J_2$ be iid uniform random elements of $[n]$, independent of $M$, we set
\begin{equation}
\tM=\Psi_{(J_1,J_2)}(M).
\end{equation}
Then $(M,\tM)$ is an exchangeable pair of rrd matrices by \Cref{lem_reflection}. 
We denote the sampled $2\times 2$ minor of the first two rows by
$$\hat{M}:=M_{(1,2)\times(J_1,J_2)}.$$

Define the antisymmetric function $F(M_1,M_2)=n\big(\co(M_1)-\co(M_2)\big)$ on $\mM_{n,d}\times \mM_{n,d}$. 
Recall the notation	
$$\eye:=\mat[1]{0}{0}{1},\quad\quad \jay:=\mat[0]{1}{1}{0}, \quad\quad \kay:= \mat[1]{0}{1}{0}.$$
Defining
$$\event =\set{(J_1,J_2)\mbox{ is reflecting}}$$
we have 
$$\big\{\hat{M}=\kay\big\} \subset \event \subset \big\{\hat{M}=\kay\big\}\vee \big\{\hat{M}=\eye\big\}$$ 
(see the discussion under \Cref{def_reflection}), and
\begin{equation}
F(M,\tM)=n\big(\un(\hat{M}=\kay)-\un(\hat{M}=\eye)\un_\event\big).
\end{equation}
Hence
\begin{align*}
f(M) &:= \e\left(F(M,\tM)\big|M\right) \\
&= n\Big[\pro{\hat{M}=\kay\big|M} - \pro{\set{\hat{M}=\eye}\wedge \event \big|M}\Big]\\
&= g(M)+ \frac{1}{n}b(M)
\end{align*}
where we define the ``main term"
\begin{equation}
g(M):=n\big[\pr(\hat{M}=\kay|M) - \pr(\hat{M}=\eye|M)\big]
\end{equation}
and the ``error term"
\begin{align}
\frac{1}{n}b(M) &:= n\,\pro{\set{\hat{M}=\eye}\wedge \event^c\,\big|\,M}\notag \\
&= \frac{1}{n}\left|\big\{(j_1,j_2)\in \Ex_M(1,2)\times \Ex_M(2,1): (j_1,j_2)\mbox{ not reflecting}\big\}\right|.\label{gdef}
\end{align}
Let us call a pair $(j_1,j_2)\in [n]^2$ ``bad" if $(j_1,j_2)\in  \Ex_M(1,2)\times \Ex_M(2,1)$ and $(j_1,j_2)$ is not reflecting.
In other words, $(j_1,j_2)$ is bad if it satisfies conditions (1) and (2) from \Cref{def_reflection} but not (3).
We have
\begin{equation}		\label{badpairs}
b(M) = \left|\big\{\mbox{bad } (j_1,j_2)\in [n]^2 \big\}\right|.
\end{equation}
Using the identities \eqref{id23}, \eqref{id4}, we see that the main term $g(M)$ is simply a shift of $\co(M)$:
\begin{align*}
g(M) &= n\left(\frac{\co(M)}{n}\frac{n-2d+\co(M)}{n} - \frac{(d-\co(M))^2}{n^2} \right)\\
&= \co(M)-\pe^2n.
\end{align*}
Hence, if we can show that the number of bad pairs $b(M)$ is small, then we can deduce tail bounds for $\co(M)$ around the value $\pe^2n$ from tail bounds for $f(M)$.

To deduce a tail bound for $f(M)$ from \Cref{thm_chatterjee}, we must bound the quantity
\begin{align*}
v_f(M) &:= \frac12 \e\left[ |f(M)-f(\tM)||F(M,\tM)|\Big|M\right] \\
&\le \frac{n}{2} \e\left[ |\co(M)-\co(\tM)|^2\Big| M\right]  + \frac{1}2 \e\left[ |\co(M)-\co(\tM)||b(M)-b(\tM)|\Big| M\right]
\end{align*}
and so we need to control the expressions $|\co(M)-\co(\tM)|$ and $|b(M)-b(\tM)|$. 

Now $\co(M)-\co(\tM)=\un(\hat{M}=\kay)-\un(\hat{M}=\eye)\un_\event$, and since these events are disjoint,
\begin{align*}
|\co(M)-\co(\tM)| &= \un(\hat{M}=\kay)+\un(\hat{M}=\eye)\un_\event \\
&\le \un(\hat{M}=\kay)+\un(\hat{M}=\eye).
\end{align*}

Since the map $\Psi_{(J_1,J_2)}$ only alters the columns indexed by $J_1,J_2$, it follows that at most $2\ex_M(1,2)$ pairs $(j_1,j_2)\in \Ex_M(1,2)\times \Ex_M(2,1)$ either become or cease to be reflecting under the application of $\Psi_{(J_1,J_2)}$, whence
\begin{align}		\label{gbound1}
|b(M)-b(\tM)| &\le 2\ex_M(1,2)\un_\event  \\
& \le 2\hat{d}\un_\event
\end{align}
where we used \eqref{exub} in the second line.
Combining these bounds with the identities \eqref{id23}-\eqref{id4}, 
\begin{align}
v_f(M) &\le  \frac{n}{2} \e\left[ |\co(M)-\co(\tM)|^2\Big| M\right]   + \hat{d}\e\left[ |\co(M)-\co(\tM)|\Big| M\right]	\notag\\
&\le \left(\frac{n}{2}+\hat{d}\right)\Big[\pr(\hat{M}=\kay|M)+\pr(\hat{M}=\eye|M)\Big]	\notag\\
&\le n\left[\frac{\co(M)(n-2d+\co(M))}{n^2}+ \frac{(d-\co(M))^2}{n^2}\right] 	\notag\\
&= \co(M)-\frac{d^2}{n}+ \frac{2}{n}(d-\co(M))^2 	\notag\\
&\le f(M)+\frac{2}{n}\hat{d}^2		\label{vfbound}
\end{align}
where in the last line we used \eqref{exub} and
\begin{align*}
\co(M) - \frac{d^2}{n}  &= g(M) \\
&= f(M) -\frac{1}{n} b(M) \\
& \le f(M)
\end{align*}
since $b(M)\ge 0$.
Applying \Cref{thm_chatterjee} with constants
$$K_1= \frac{2}{n}\hat{d}^2, \quad K_2=1,$$
we conclude that for any $t\ge 0$, 
\begin{align*}
\pro{\co(M) - \pe^2 n \ge t} &\le \pro{\co(M)-\pe^2n+ \frac{1}{n}b(M)\ge t} \\
&= \pro{f(M)\ge t} \\
&\le \expo{ - \frac{ t^2/2}{\frac{2}{n}\hat{d}^2 + t}}
\end{align*}
where we again used that $b(M)\ge 0$. 
The terms in the denominator are balanced by scaling $t=\eps \hat{\pe}^2n$, where we recall
\begin{equation}
\hat{\pe}:=\hat{d}/n = \frac{1}{n}\big(\min(d,n-d)\big)
\end{equation}
giving the desired bound
\begin{align*}
\pro{\co(M) - \pe^2 n \ge \eps \hat{\pe}^2 n} \le \expo{-\frac{\eps^2}{4+2\eps} \hat{\pe}^2 n}.
\end{align*}
\end{proof}

\section{Uniform control on codegrees}		\label{sec_codeg2}

In this section we complete the proof of part (1) of \Cref{thm_main}.

In the previous section, we could pass from control on the upper tail of $f(M)$ to control on the upper tail of 
$$\co(M)=\pe^2n+  f(M)-\frac1n b(M) $$
using the fact that the number of bad pairs $b(M)$ (defined in \eqref{badpairs}) is non-negative. 
In order to control the lower tail of $\co(M)$, we will need to improve on the trivial upper bound 
$$b(M)\le |\ex_M(1,2)|^2\le \hat{d}^2$$
(from monotonicity and \eqref{exub}).
In this section we show that $b(M)\le \eps \hat{d}^2$ with high probability for $\eps>0$ small.
A key ingredient will be the control on the upper tail of the codegrees obtained in the previous section.

Part (1) of \Cref{thm_main} follows from substituting $\eps=\eta\frac{\max(p,1-p)}{\min(p,1-p)}$ in the following proposition.

\begin{proposition}[Uniform bounds on codegrees]		\label{thm_codeg}
For any $\eps\ge 0$, 
\begin{equation}		\label{cocontrol}
\pro{\exists  \set{i_1,i_2}\subset [n]: \; \left| \harp{\co}_\Gamma(i_1,i_2)-\frac{d^2}{n}\right|\ge \eps\frac{\hat{d}^2}{n}} \le C_1 n^2\hat{d}^2\expo{-c\eps \hat{d}}+ C_2 n^2\expo{-\frac{c\eps^2}{1+\eps}\frac{\hat{d}^2}{n}}
\end{equation}
where $\hat{d}=\min(d, n-d)$, and $C_1,C_2, c>0$ are absolute constants. If $\eps\ge 1$ we may take $C_1=0$. 

\end{proposition}

\begin{remark}
For $\eps\in (0,1)$ fixed independent of $n$ and $\hat{d}=\omega(\log n)$ the first term on the right hand side of \eqref{cocontrol} is of lower order. 
The second term matches the bound \eqref{ercocontrol} for \er digraphs, up to the constants in the exponential. 
\end{remark}

\begin{proof}

For $i_1,i_2\in [n]$ distinct, define
\begin{equation}
b_{(i_1,i_2)}(M)= \Big|\Big\{(j_1,j_2)\in \Ex_M(i_1,i_2)\times \Ex_M(i_2,i_1): (j_1,j_2)\mbox{ not reflecting}\Big\}\Big|
\end{equation}
so that in the notation of \eqref{gdef} we have $b(M)=b_{(1,2)}(M)$. 
By row-exchangeability it suffices to get control on $b_{(1,2)}(M)$ and apply a union bound over all $(i_1,i_2)\in [n]^2$.

\begin{lemma}		\label{lem_gbound}
For any $\lambda>0$, 
$$\pr\Big\{\, b_{(1,2)}(M)\ge \lambda\hat{d}  \,\Big\} \ll  \hat{d}^2 \expo{-c\hat{d}} + \hat{d} e^{-c\lambda}.$$
\end{lemma}

\begin{proof}

Defining the subsets of $[n]^2$
\begin{align*}
Q(M) &:= \Ex_M(1,2)\times \Ex_M(2,1),\\
B(M) &:= \set{(j_1,j_2) \mbox{ not reflecting}}
\end{align*}
we have 
$$b(M)= |Q(M)\cap B(M)|.$$
Denote
$$\x= |\Ex_M(1,2)|=|\Ex_M(2,1)|.$$
Now we decompose $b(M)$ as a sum of $\x$ terms, each of which can be expressed as a sum of independent indicators. 
We enumerate the elements of $\Ex_M(1,2), \Ex_M(2,1)$ in increasing order as $j^+_1<\dots<j^+_{\x}$ and $j^-_1<\dots<j^-_{\x}$, respectively. 
For each $s\in [0, \x-1]$, define
$$Q_s(M)=\set{(j^+_m,j^-_{m+s}): m\in [\x]}$$
with the sum $m+s$ understood to be mod $\x$, and put
$$b_s(M)=|Q_s(M)\cap B(M)|$$
so that 
\begin{equation}	\label{gdecomp}
b(M)=\sum_{s=0}^{\x -1} b_s(M).
\end{equation}

Fix $0\le s\le \x-1$. 
We now construct an exchangeable pair $(M,\tM)$ by resampling a certain subset of the entries of $M$.
For each element $(j^+,j^-)\in Q_s(M)$ write 
\begin{align*}
\mI(j^+,j^-) &= [3,n]\cap\big(\Ex_{M^\tran}(j^+,j^-)\cup \Ex_{M^\tran}(j^-,j^+)\big)\\
&= \big\{i\in [3,n]: M(i,j^+)+M(i,j^-)=1\big\}.
\end{align*}
We form the pair $(M,\tM)$ by first drawing $M\in \mM_{n,d}$ uniformly, then forming $\tM$ by independently and uniformly resampling the $k$ sub-matrices
\begin{equation}	\label{submats}
\set{ \tM_{\mI(j^+,j^-)\times(j^+,j^-)}}_{(j^+,j^-)\in Q_s(M)}
\end{equation}
conditional on all other entries of $M$.
We can do this resampling independently since our conditioning has already fixed all of the row and column sums of each of these sub-matrices.  
For exchangeability it is important to note that $Q_s(M)=Q_s(\tM)$, as this set is determined by the first two rows of $M$, which are not resampled.

We will restrict to an event on which we have an upper bound on codegrees. 
Let 
\begin{equation}
\good_{s} = \bigwedge_{(j^+,j^-)\in Q_s(M)}\set{\co_{M^\tran}(j^+, j^-) \le \left(\frac{1+p}2\right)d}
\end{equation}
enforcing a slight improvement on the deterministic upper bound $\co_{M^\tran}(j^+,j^-)\le d$. 
By a union bound and \Cref{prop_codeg} (taking $\eps$ to be a small multiple of $n/\hat{d}$) we have
\begin{equation}
\pro{ \good_s} \ge 1- \x e^{-c\hat{d}} \ge 1-\hat{d}e^{-c\hat{d}}.
\end{equation}
Note that $\good_s$ holds for $M$ if and only if it holds for $\tM$, since the resampling does not change the value of 
$\co_{M^\tran}(j^+,j^-)$ for any $(j^+,j^-)\in Q_s(M)$.

Conditional on $M$, from the joint independence of the sub-matrices \eqref{submats} we see that $b_s(\tM)$ is a sum of independent indicators. 
Hence, we can control the upper tail of $b_s(\tM)$ using Bernstein's inequality, once we have estimates on $\e \big[b_s(\tM)\big| M\big]$. 
We will then deduce the desired bound on $b(M)$ through the decomposition \eqref{gdecomp} and a union bound.

To estimate $\e \big[b_s(\tM)\big| M\big]$ we have the following
\begin{claim}		\label{claim_bad}
For each $s\in [0,\x-1]$ 
and $(j^+,j^-)\in Q_s(M)$, 
\begin{equation}	\label{indicatorp}
\pr\big[(j^+,j^-)\in B(\tM)\big|M\big]\un_{\good_s}\ll1/\hat{d}.
\end{equation}

\end{claim}

Let us assume this claim for now. 
Restricting to $\good_s$, from \eqref{indicatorp} we have 
$$\e \big[b_s(\tM)\big|M\big]\un_{\good_s} \ll \x/ \hat{d} \le 1$$
 for each $s\in [0,\x-1]$, where we used \eqref{exub}. 
 Moreover, since $b_s(\tM)|M$ is a sum of independent indicator variables, from Bernstein's inequality we have that for any $\lambda> 0$, 
$$\pro{b_s(\tM)\ge \lambda \,\Big|\,M} \un_{\good_s}\ll \exp(-c\lambda)$$
and so
\begin{align*}
\pro{ b_s(M) \ge \lambda} &=\pro{ b_s(\tM) \ge \lambda} \\
&\le  \pro{\good_s^c} + \e \pro{ b_s(\tM) \ge \lambda \big| M} \un_{\good_s} \\
&\ll \hat{d}\expo{-c\hat{d}} + e^{-c\lambda}.
\end{align*}
By pigeonholing and a union bound it follows that
\begin{align*}
\pro{ b(M) \ge \lambda \hat{d} } &\le \hat{d} \;\pro{ b_0(M) \ge \lambda}  \\
&\ll \hat{d}^2 \expo{-c\hat{d}} + \hat{d} e^{-c\lambda}.
\end{align*}

It remains to establish \Cref{claim_bad}. 
Fix $s$ and 
$(j^+,j^-)$ as in the claim. 
Consider the walk $w=w_{(j^+,j^-)}:[n]\rightarrow \Z$ associated to the pair of columns $(X_{j^+}, X_{j^-})$ as defined in \eqref{def_walk}. 
Since $(j^+,j^-)\in \Ex_M(1,2)\times\Ex_M(2,1)$ by assumption, we have $M_{(1,2)\times(j^+,j^-)}=\eye$, and so $w(1)=1$ and $w(2)=0$. 
The event that $(j^+,j^-)\in B(M)$ is the event that there is no $i\in [3,n]$ such that $w(i)=1$, i.e.\ that $w$ is ``non-crossing" in this range. 
Let us condition on the number $r$ of steps taken to the right by $w$; by our restriction to $\good_s$ we have 
$$r=d-\co_{M^\tran}(j^+,j^-) \ge \frac12 (1-p)d.$$
Conditional on $r$, in the randomness of the resampling of $(X_{j^+}, X_{j^-})$ we have that every ordering of the $r-1$ left steps of $w$ and $r-1$ right steps in the range $[3,n]$ is equally likely. 
There are ${2(r-1)\choose r-1}$ such orderings, while the number of these giving non-crossing walks is the Catalan number
$$\frac{1}{r}{2(r-1)\choose r-1}.$$
It follows that under the resampling, the probability that $(j^+,j^-)\in B(M)$ is 
$$\frac{1}{r} \ll \frac{1}{(1-p)d} \ll \frac{1}{\hat{d}}.$$ Undoing the conditioning on $r$, the claim follows.
\end{proof}

Now we can get a good lower tail estimate on $\co(M)$ and complete the proof of \Cref{thm_codeg}.

Fix $\eps\ge0$. 
If $\eps\ge 1$ then the result already follows from \Cref{prop_codeg} and a union bound as the lower tail event is empty in this case.
Hence we may assume $\eps<1$.
We may further assume that $\eps$ is sufficiently small by adjusting the constant $c$ in the statement of the theorem. 

For $\lambda\ge 0$ and $i_1,i_2\in [n]$ distinct, let 
$$\bad_{(i_1,i_2)}(\lambda)=\set{b_{(i_1,i_2)}(M)\ge \lambda\hat{d}}$$
and 
$$\good(\lambda):= \bigwedge_{i_1\ne i_2\in [n] }\bad_{(i_1,i_2)}(\lambda)^c$$
From \Cref{lem_gbound} and a union bound, we have
\begin{equation}
\pro{\good(\lambda)^c} \ll n^2\Big(\hat{d}^2 \expo{-c\hat{d}} + \hat{d} e^{-c\lambda} \Big).
\end{equation}
Restricting to the good event, we can bound
\begin{align*}
\pro{\good(\lambda)\wedge\set{\co(M) \le (1-\eps)\hat{\pe}^2n}} &= \pro{\good(\lambda)\wedge\set{f(M)\le -\eps\hat{\pe}^2n+\frac1n b(M)}} \\
&\le  \pro{f(M)\le -\hat{\pe}(\eps\hat{d} - \lambda)}.
\end{align*}
Taking $\lambda=\eps \hat{d}/2$ and applying \Cref{thm_chatterjee} (with the bound \eqref{vfbound} on $v_f(M)$) the last quantity is bounded by $\expo{-c\eps^2\hat{\pe}^2n}$.
Putting it all together, denoting
$$\bad = \bigvee_{\set{i_1,i_2}\subset [n]} \set{\big| \co_M(i_1,i_2)-\pe^2n\big|\ge \eps\hat{\pe}^2n}$$
we have
\begin{align*}
\pro{\bad} &\le \pro{\good(\lambda)^c} + \sum_{\set{i_1,i_2}\subset[n]} \pro{\good(\lambda)\wedge\Big\{\big| \co_M(i_1,i_2)-\pe^2n\big|\ge \eps\hat{\pe}^2n\Big\}}\\
&\ll  n^2\hat{d}^2\expo{-c\eps \hat{d}} + n^2 \expo{-\frac{c\eps^2}{1+\eps}\frac{\hat{d}^2}{n}}.
\end{align*}
\end{proof}

\section{Concentration of edge counts}		\label{sec_edge}

In this section we prove part (2) of \Cref{thm_main}, using \Cref{thm_chatterjee} with the switching coupling of \Cref{lem_switching}.
A crucial ingredient will be the control on codegrees enforced by restriction to the event $\good^{\co}(\eta)$. 
The reader may wish to read the simpler proof of \Cref{thm_edgeperm} in \Cref{sec_perm} first, as it uses a similar switching on permutation matrices, but does not require restriction to the event $\good^{\co}(\eta)$. 

Fix $A,B\subset [n]$, and let us denote $|A|=a$, $|B|=b$. 
Without loss of generality we may assume 
\begin{equation}	\label{assumeab}
a+b\le n.
\end{equation}
Indeed, as noted in \eqref{eventid}, for any $t\in \R$,
$$\set{e_M(A,B)-\mu(A,B)\ge t} = \set{e_M(A^c,B^c) - \mu(A^c,B^c) \ge t}.$$
Hence, if we establish the claim assuming $a+b\le n$, then for the case that $a+b>n$ we can apply the claim to $(A^c,B^c)$ rather than $(A,B)$.
Under assumption \eqref{assumeab} we have
\begin{equation}
\hat{\mu}(A,B) = \pe\min\big[ab, (n-a)(n-b)\big] = \pe ab = \mu(A,B).
\end{equation}

We define an exchangeable pair of rrd matrices $(M,\tM)$ as follows. Draw $M$, and sample 
$$I_1\in A,\; I_2\in [n]\setminus A,\;  J_1\in B,\; J_2\in [n]\setminus B$$
uniformly from their respective ranges, independently of each other and of $M$. 
Conditional on $M,I_1,I_2,J_1,J_2$, form $\tM$ by performing a switching at the minor $(I_1,I_2)\times (J_1,J_2)$.
$(M,\tM)$ is an exchangeable pair by \Cref{lem_switching}. 

Let us denote
$$K_{ab}:=a(n-a)b(n-b).$$
Define the antisymmetric function $F:\mM_{n,d}\times\mM_{n,d}\rightarrow \R$ by 
$$F(M_1,M_2)=K_{ab}\Big[e_{M_1}(A,B) - e_{M_2}(A,B)\Big].$$
Denote the sampled $2\times 2$ minor $M_{(I_1,I_2)\times(J_1,J_2)}$ by $\hat{M}$.
We have
$$\event:=\set{\mbox{$\hat{M}$ is switchable}} = \set{\hat{M}=\eye}\vee \set{\hat{M}=\jay}$$
and
$$F(M,\tM)=K_{ab}\big[\un(\hat{M}=\eye)-\un(\hat{M}=\jay)\big].$$
We have
\begin{align}
f(M)&:= \e\left[F(M,\tM)\big|M\right] \notag\\
&= K_{ab}\Big[\pro{\hat{M}=\eye\,\Big|\,M}-\pro{\hat{M}=\jay\,\Big|\,M} \Big]\label{fd1}\\
&=\sum_{(i_1,i_2)\in A\times A^c} \bigg(|\Ex_M(i_1,i_2)\cap B| |\Ex_M(i_2,i_1)\cap B^c| 
\notag\\
&\hspace{4cm}- |\Ex_M(i_2,i_1)\cap B||\Ex_M(i_1,i_2)\cap B^c|\bigg)		\label{fd2}
\end{align}
(with $A^c=[n]\setminus A$, $B^c=[n]\setminus B$). 

Before proceeding to control the expression $v_f(M)$ from \Cref{thm_chatterjee}, let us show how $f(M)$ is related to $e_M(A,B)$. 
Recalling the notation
\begin{align*}
\ex_M(i_1,i_2)&=d-\co_M(i_1,i_2) \\
&=|\Ex_M(i_1,i_2)| = |\Ex_M(i_2,i_1)|
\end{align*}
we re-express the summand in \eqref{fd2} as 
\begin{align*}
&\quad|\Ex_M(i_1,i_2)\cap B|\Big(\ex_M(i_1,i_2)-|\Ex_M(i_2,i_1)\cap B|\Big) \\
&\quad\quad\quad\quad\quad\quad\quad\quad- |\Ex_M(i_2,i_1)\cap B|\Big(\ex_M(i_1,i_2)-|\Ex_M(i_1,i_2)\cap B|\Big) \\
&\quad\quad\quad\quad= \ex_M(i_1,i_2)\big(|\Ex_M(i_1,i_2)\cap B|-|\Ex_M(i_2,i_1)\cap B|\big)  \\
&\quad\quad\quad\quad =  \ex_M(i_1,i_2)\big(|\mN_M(i_1)\cap B|-|\mN_M(i_2)\cap B|\big).
\end{align*}
Putting this in \eqref{fd2} we have
\begin{equation}
f(M) = 
\sum_{(i_1,i_2)\in A\times A^c}\ex_M(i_1,i_2)\Big(|\mN_M(i_1)\cap B|-|\mN_M(i_2)\cap B|\Big).		\label{fd3}
\end{equation}
On $\good^{\co}(\eta)$ the quantities $\ex_M(i_1,i_2)$ all lie in $p(1-p)n[1-\eta,1+\eta]$.
Writing 
$$\ex_M(i_1,i_2)=p(1-p)n +\big(\ex_M(i_1,i_2)-p(1-p)n\big)$$ 
we can express
\begin{equation}	\label{fd4}
f(M)=f_1(M)+f_2(M)
\end{equation}
where we define the ``main term"
\begin{align*}
f_1(M) &:= p(1-p)n \sum_{(i_1,i_2)\in A\times A^c} |\mN_M(i_1)\cap B|-|\mN_M(i_2)\cap B| \\
&= p(1-p)n\big[(n-a)e_M(A,B)-ae_M(A^c,B)\big] \\
&=  p(1-p)n^2 \big[e_M(A,B)-\pe ab\big].
\end{align*}
and the ``error term"
\begin{align*}
f_2(M):= \sum_{(i_1,i_2)\in A\times A^c}\big(\ex_M(i_1,i_2)-p(1-p)n\big)\Big(|\mN_M(i_1)\cap B|-|\mN_M(i_2)\cap B|\Big).
\end{align*}

We now show that $f_2(M)$ is small on $\good^{\co}({\eta})$ if ${\eta}$ is sufficiently small, so that on this event $f(M)\approx f_1(M)$, a scaling and centering of $e_M(A,B)$.
Indeed, letting ${\eta}>0$ to be chosen later, 
\begin{align}
|f_2(M)|\un_{\good^{\co}({\eta})} &\le {\eta}p(1-p)n\sum_{(i_1,i_2)\in A\times A^c}|\mN_M(i_1)\cap B|+|\mN_M(i_2)\cap B| \notag\\
& =  {\eta}p(1-p)n \big[\, (n-a)e_M(A,B) + a e_M(A^c,B)\,\big] \notag\\
&\le {\eta}p(1-p)n^2\left[\, e_M(A,B)-\mu(A,B) + 2\mu(A,B) \,\right] \notag\\
&={\eta} \big[\, f_1(M) + 2p(1-p)n^2 \mu\,\big]		\label{f2bd}
\end{align}
where in the third line we added and subtracted $\mu(A,B)=dab/n$ and used $e_M(A^c,B) \le db.$

Now we will bound the quantity
$$v_f(M):=\frac{1}{2}\e\left[\big|\big(f(M)-f(\tM)\big)F(M,\tM)\big|\Big|M\right]$$
from \Cref{thm_chatterjee}.
First we bound $|f(M)-f(\tM)|$ by considering the expression \eqref{fd2}. 
Since $M$ and $\tM$ only differ on the $(I_1,I_2)\times (J_1,J_2)$ minor, the only summands in \eqref{fd2} that do not cancel in $f(M)-f(\tM)$ have indices in the set
$$\mI:=\Big\{(i_1,i_2)\in A\times A^c: \mbox{ either $i_1=I_1$ or $i_2=I_2$ (or both)}\Big\}.$$ 
Now note that for any $(i_1,i_2)\in \mI$, 
$$\big||\Ex_M(i_1,i_2)\cap B| - |\Ex_{\tM}(i_1,i_2)\cap B|\big|\le \un_{\event}$$
since this set only changes (possibly) if $\hat{M}$ is switchable. 
We have the same bound for the pair neighborhoods $\Ex_M(i_2,i_1)$ and with $B$ replaced by $B^c$. 
Using these bounds with \eqref{fd2} we have 
\begin{align*}
|f(M)-f(\tM)| &\le \sum_{(i_1,i_2)\in \mI} \un_\event 
\Big[ |\Ex_M(i_1,i_2)\cap B| + |\Ex_{\tM}(i_2,i_1)\cap B^c| \\
&\quad \hspace{3cm}+ |\Ex_{\tM}(i_2,i_1)\cap B| + |\Ex_M(i_1,i_2)\cap B^c| \Big] \\
&= \sum_{(i_1,i_2)\in \mI} \un_\event 
\Big[ |\Ex_M(i_1,i_2)| + |\Ex_{\tM}(i_2,i_1)| 	\Big] \\
&\le 2|\mI| \hat{d}\un_\event \\
&= 2n\hat{d}\un_\event
\end{align*}
where in the third line we applied the upper bound \eqref{exub} for $M$ and $\tM$.

Now since 
$$\frac{1}{K_{ab}}|F(M,\tM)|=|\un(\hat{M}=\eye)-\un(\hat{M}=\jay)|\le\un_\event$$
we have
\begin{align*}
v_f(M) &\le n\hat{d}K_{ab}\pr(\event|M).
\end{align*}
We want to show that $f(M)$ is ``self bounding" in the sense that we can control
$v_f(M)$ by an expression of the form $K_1+K_2f(M)$ for some constants $K_1,K_2>0$ (possibly depending on $n,d,a,b$). 
Since 
$$\pr(\event|M)=\pro{\hat{M}=\eye\,\Big|\,M}+\pro{\hat{M}=\jay\,\Big|\,M}$$
we have
\begin{align}
v_f(M) &\le n\hat{d}K_{ab}\left[\pro{\hat{M}=\eye\,\Big|\,M}+\pro{\hat{M}=\jay\,\Big|\,M}\right]	\notag\\
& = n\hat{d}\left[ f(M)+2K_{ab}\,\pro{\hat{M}=\jay\,\Big|\,M}\right]	\label{self}
\end{align}
where we have used \eqref{fd1} in the second line. 
Writing
\begin{align*}
\pro{\hat{M}=\jay\,\Big|\,M} & = \frac{1}{K_{ab}}\sum_{(i_1,i_2)\in A\times A^c}|\Ex_M(i_2,i_1)\cap B||\Ex_M(i_1,i_2)\cap B^c|
\end{align*}
we can crudely bound $|\Ex_M(i_1,i_2)\cap B^c|\le \hat{d}$ and $|\Ex_M(i_2,i_1)\cap B|\le |\mN_M(i_2)\cap B|$ (from monotonicity) to get
\begin{align*}
K_{ab}\pro{\hat{M}=\jay\,\Big|\,M} &\le \hat{d}a  \sum_{i_2\in A^c}  |\mN_M(i_2)\cap B| \\
&=\hat{d}a e_M(A^c,B) \\
&\le \hat{d}a (db) \\
&= \hat{d}n \mu(A,B).
\end{align*}
Combining the last line with \eqref{self} we conclude
\begin{equation}	\label{vfbound}
v_f(M)\le n\hat{d}\big[ f(M)+ 2n\hat{d}\mu\big].
\end{equation}

From \eqref{f2bd}, on $\good^{\co}({\eta})$ we have
\begin{align*}
f(M) &= f_1(M) +f_2(M) \\
&\ge f_1(M) - |f_2(M)| \\
&\ge (1-{\eta})f_1(M) - 2{\eta}p(1-p)n^2 \mu.
\end{align*}
It follows that for ${\eta},t\ge0$ fixed,  
\begin{align}
\pro{ \good^{\co}({\eta}) \wedge \Big\{e_M(A,B)-\mu\ge t\Big\} }
&= \pro{\good^{\co}({\eta})\wedge \Big\{ f_1(M)\ge p(1-p)n^2 t \Big\} }		\notag\\
&\le \pr\Big(  f(M) \ge \big[(1-{\eta}) t -  2{\eta} \mu \big]p(1-p)n^2  \Big)	\label{uptl}
\end{align}
and
\begin{align}
\pro{ \good^{\co}({\eta}) \wedge \Big\{e_M(A,B)-\mu\le -t\Big\} }
&= \pro{\good^{\co}({\eta})\wedge \Big\{ f_1(M)\le -p(1-p)n^2 t \Big\} }		\notag \\
&\le \pr\Big(  f(M) \le - \big[t - 2{\eta} \mu \big]p(1-p)n^2   \Big).	\label{dntl}
\end{align}

Let us scale $t=\tau\mu$. 
If we take ${\eta}\le \min(1/4, \tau/8)$, then from \eqref{uptl}
\begin{equation}
\pro{ \good^{\co}({\eta}) \wedge \Big\{e_M(A,B)-\mu\ge \tau\mu\Big\} }
\le \pro{ f(M) \ge \frac\tau2p(1-p)n^2 \mu }		\label{tochat}
\end{equation}
Now we may apply \Cref{thm_chatterjee} to the right hand side of \eqref{tochat} with 
\begin{equation}	\label{chosenK}
K_1=2n^2\hat{d}^2\mu, \quad K_2 = n\hat{d}
\end{equation}
from \eqref{vfbound} to bound 
\begin{align*}
\pro{ \good^{\co}({\eta}) \wedge \Big\{e_M(A,B)-\mu\ge \tau\mu\Big\} }
&\le \expo{ -\frac{ \big(\frac12 \tau p(1-p)n^2 \mu\big)^2}{2n\hat{d} \big( 2n\hat{d}\mu + \frac12\tau p(1-p)n^2 \mu \big)}}\\
& = \expo{ -\frac{ \tau^2 \mu}{8\frac{\hat{d}}{p(1-p)n} \big( 2\frac{\hat{d}}{p(1-p)n} + \frac12\tau  \big)}}\\
&\le \expo{ -\frac{ \tau^2 \mu}{64+ 8\tau  }}\\
\end{align*}
where in the last line we used that
$$\frac{\hat{d}}{p(1-p)n} = \frac{\min(d,n-d)}{\frac{1}{n}d(n-d)} \le 2.$$

The lower tail is obtained similarly from \eqref{dntl} and \Cref{thm_chatterjee} (and only requiring that we take ${\eta}\le \tau/4$).

\subsection*{Acknowledgement}
The author thanks the anonymous referees for various corrections and helpful suggestions to improve the paper.

\bibliographystyle{abbrv}
\bibliography{biblio_discrep.bib}

\end{document}